\documentclass{amsart}
\usepackage{amsmath,amssymb}
\usepackage[all]{xy}

\newcommand{\hg}{\hat g}

\renewcommand{\P}{\mathbb P}
\newcommand{\Z}{\mathbb Z}
\newcommand{\Q}{\mathbb Q}
\newcommand{\C}{\mathbb C}

\newcommand{\cal}{\mathcal}
\newcommand{\cI}{\mathcal I}
\newcommand{\cO}{\mathcal O}
\newcommand{\cR}{\mathcal R}

\renewcommand{\t}{\widetilde}

\newcommand{\bV}{\overline{V}}

\newcommand{\m}{\mathfrak m}

\DeclareMathOperator{\supp}{Supp}
\DeclareMathOperator{\di}{div}
\DeclareMathOperator{\Di}{Div}
\DeclareMathOperator{\Cl}{Cl}

\DeclareMathOperator{\lcm}{lcm} 
 
\DeclareMathOperator{\mult}{mult}
\DeclareMathOperator{\emb}{embdim}

\DeclareMathOperator{\spec}{Spec}

\DeclareMathOperator{\cf}{cff}

\renewcommand{\:}{\colon}

\newcommand{\Fl}[1]{\left\lfloor #1 \right\rfloor}
\newcommand{\Ce}[1]{\left\lceil #1 \right\rceil}
\newcommand{\gen}[1]{\langle #1 \rangle}

\newcommand{\defset}[2]{{\left\{\left.#1\,\right| \,#2  \right\}}}
\newcommand{\V}{(V,o)}

\newcommand{\dis}{\displaystyle}

 \theoremstyle{plain}
\newtheorem{thm}{Theorem}[section]
\newtheorem{cor}[thm]{Corollary}
\newtheorem{lem}[thm]{Lemma}
\newtheorem{prop}[thm]{Proposition}

 \theoremstyle{definition}
\newtheorem{defn}[thm]{Definition}
\newtheorem{nota}[thm]{Notation}
\newtheorem{ex}[thm]{Example}

\newtheorem{prob}[thm]{Problem}

\newtheorem*{acknowledgement}{Acknowledgement}

\theoremstyle{remark}

\newtheorem{rem}[thm]{Remark}

\numberwithin{equation}{section}

\newcommand{\thmref}[1]{Theorem~\ref{#1}}
\newcommand{\lemref}[1]{Lemma~\ref{#1}}

\newcommand{\proref}[1]{Proposition~\ref{#1}}
\newcommand{\remref}[1]{Remark~\ref{#1}}

\newcommand{\defref}[1]{Definition~\ref{#1}}
\newcommand{\notref}[1]{Notation~\ref{#1}}

\newcommand{\exref}[1]{Example~\ref{#1}}
\newcommand{\figref}[1]{Figure~\ref{#1}}
\newcommand{\sref}[1]{Section~\ref{#1}}

\begin{document}

\title[Brieskorn complete intersections]{Weighted homogeneous surface singularities homeomorphic to Brieskorn complete intersections}

\author{Tomohiro Okuma}
\address{Department of Mathematical Sciences, 
Yamagata University, 
 Yamagata 990-8560, Japan.}
\email{okuma@sci.kj.yamagata-u.ac.jp}
\thanks{This work was partially supported by JSPS KAKENHI Grant Number 17K05216}
\subjclass[2010]{Primary 32S25; Secondary 14J17, 32S05, 14B05}

\keywords{Surface singularities, weighted homogeneous singularities, Brieskorn complete intersections, geometric genus, maximal ideal cycles}
\begin{abstract}
For a given topological type of a normal surface singularity, there are various types of complex structures which realize it.
We are interested in the following problem:
Find the maximum of the geometric genus and a condition for that the maximal ideal cycle coincides with the fundamental cycle on the minimal good resolution.
In this paper, we study weighted homogeneous surface singularities homeomorphic to Brieskorn complete intersection singularities from the perspective of the problem.
\end{abstract}

\maketitle


\section{Introduction}
The topological type of a normal surface singularity is determined by its resolution graph (\cite{neumann.plumbing}).
For a given resolution graph of a normal surface singularity, there are various types of complex structures which realize it.
We are interested in finding the upper (resp. lower) bound of basic invariants (e.g., the geometric genus), and in understanding the complex structures which attain their maximum (resp. minimum).

Let $\V$ be a normal complex surface singularity with minimal good resolution $X\to V$ and let $\Gamma$ be the resolution graph of $\V$.
As noticed above, the topological invariants of $\V$ are precisely the invariants of $\Gamma$.
In this paper, we consider the geometric genus $p_g\V=\dim H^1(\cO_X)$ and the maximal ideal cycle $M_X$ on $X$. 
In general, these invariants cannot be determined by $\Gamma$ and it is difficult to compute them.
By the definition (\defref{d:cycles}),  the fundamental cycle $Z_X$  on $X$ is determined by $\Gamma$ and the inequality $M_X\ge Z_X$ holds.
The fundamental problem we wish to explore is the following.
\begin{prob}
Let $p_g(\Gamma)$ denote the maximum of the geometric genus over the normal surface singularities with resolution graph $\Gamma$.
\begin{enumerate}
\item Find the value $p_g(\Gamma)$ and conditions for $M_X=Z_X$.
\item Describe the properties and invariants of a singularity $\V$ with $p_g\V=p_g(\Gamma)$ or $M_X=Z_X$.
\end{enumerate}
\end{prob}
It is known that in a complex analytic family of the resolution space $X$ preserving $\Gamma$ (cf. \cite{la.lift}), the dimension of the cohomology of the structure sheaf is upper semicontinuous.
So, we expect the singularities $\V$ with  $p_g\V=p_g(\Gamma)$ may have some kind of nice structure.

The equality $M_X=Z_X$ holds for rational singularities  (\cite{artin.rat}), minimally elliptic singularities (\cite{la.me}), and hypersurfaces $z^n=f(x,y)$ with certain conditions (\cite{dixon}, \cite{tomaru-Kodaira}).
We have an explicit condition for the equality $M_X=Z_X$ for Brieskorn complete intersection singularities (\cite{K-N}, \cite{MO}); the result is generalized to Kummer coverings over weighted homogeneous normal surface singularities in \cite{TT}.
The upper bound of $p_g$ has been also studied by several authors (e.g., \cite{yau.max}, \cite{tomari.ell}, \cite{tomari.max}, \cite{nem.lattice}, \cite{N-Sig}); the ``rational trees'' $\Gamma$ whose $p_g(\Gamma)$ can be obtained from $\Gamma$ are listed in \cite[1.7]{no-2cusp}.
In \exref{e:mpg} of the present paper, we shall introduce the weighted homogeneous singularities of {\em hyperelliptic type} for which $p_g(\Gamma)$ is easily computed.
Since $p_g\V=\dim H^0(\cO_X)/H^0(\cO_X(-Z_{K_X}))$ for numerically Gorenstein singularity, 
where $Z_{K_X}$ is the canonical cycle (\defref{d:cycles}), it might be natural to expect that there is a correlation between the properties $p_g\V=p_g(\Gamma)$ and $M_X=Z_X$.
In fact, when $\V$ is a numerically Gorenstein elliptic singularity (this is characterized by $\Gamma$), we have that $p_g\V=p_g(\Gamma)$ if and only if $\V$ is a Gorenstein singularity with $M_X=Z_X$ (\cite[5.10]{o.numGell}, \cite{yau.max}, \cite{nem.ellip}); in this case, $p_g(\Gamma)$ coincides with the length of the elliptic sequence.
However, in \cite{no-2cusp}, we found an example such that the equality $p_g=p_g(\Gamma)$ is realized by both a Gorenstein singularity with $M_X>Z_X$ and a non-Gorenstein singularity with $M_X=Z_X$.
In \sref{s:exBCI}, we give an example which shows that the condition  $M_X=Z_X$ cannot control $p_g$.

In this paper, we study normal surface singularities homeomorphic to Brieskorn complete intersection singularities from the perspective of our problem above.
First suppose that $V$ is a complete intersection given as follows:
\[
V=\defset{(x_i)\in \C^m}
{q_{i1}x_1^{a_1}+\cdots +q_{im}x_{m}^{a_{m}} =0,
\quad i=3,\dots , m} 
\quad (q_{ij}\in \C).
\]
The resolution graph of the singularity $\V$ is determined by the integers $a_1, \dots, a_m$ (\thmref{t:BCImain}).
We denote it by $\Gamma(a_1, \dots, a_m)$.
Using the Pinkham-Demazure divisor $D$ on the central curve $E_0$ of the exceptional set $E\subset X$, the homogeneous coordinate ring $R$ of $V$ is represented as $R=\bigoplus_{k\ge 0}H^0(\cO_{E_0}(D_k))T^k$ (see \sref{ss:WH}).
We study arithmetic properties of the numerical invariants arising from the topological type in terms of the divisors $D_k$ on $E_0$.
For this purpose, we employ the monomial cycles (cf. \cite{o.pg-splice}) to connect the numerical information of the divisors $D_k$ and the complex analytic functions on $X$; note that monomial cycles play an important role in the study of invariants of splice quotients (\cite{o.pg-splice}, \cite{nem.coh-sq}).
For example, we show that $H^0(\cO_{E_0}(D_k))\ne 0$ if and only if $\deg D_k$ is a member of a certain semigroup, and that $D_k\sim D_{k'}$ if and only if $\deg D_k = \deg D_{k'}$ (see \proref{p:monomials}, \thmref{t:ehg}).
Applying these results, we obtain the following (see \thmref{t:g1p_gmax}, \thmref{t:g=1MZ}).

\begin{thm}\label{t:simple}
If $\V$ is a Brieskorn complete intersection such that the central curve $E_0$ is rational or elliptic curve, then $p_g\V=p_g(\Gamma)$ and $M_X=Z_X$.
\end{thm}

Even if the singularity is not a Brieskorn complete intersection, we can apply a part of the argument on the divisors $D_k$ and prove the following  (\thmref{t:Z=M}).

\begin{thm}\label{t:int-exist}
There exists a weighted homogeneous singularity with resolution graph $\Gamma(a_1, \dots, a_m)$ such that
the maximal ideal cycle coincides with the fundamental cycle on the minimal good resolution.
\end{thm}

We shall describe the property of the Pinkham-Demazure divisor corresponding to the singularity in \thmref{t:int-exist}.

If the central curve $E_0$ has genus $g\ge 2$, we cannot expect a  result similar to \thmref{t:simple}.
In fact, there may be various types of complex structures even when $g=2$. 
To show this, in \sref{s:exBCI}, we fix a resolution graph $\Gamma=\Gamma(2,3,3,4)$, which is the simplest one in a sense, and investigate the singularities having this graph.
Any Brieskorn complete intersection singularity with this graph satisfies neither $p_g\V=p_g(\Gamma)$ nor $M_X=Z_X$.
Assume that $\V$ is a weighted homogeneous surface singularity with resolution graph $\Gamma$.
We prove that $\V$ satisfies $p_g\V=p_g(\Gamma)$ if and only if it is hyperelliptic type, and show that such a singularity is a complete intersection, which is a double cover of a rational double point of type $A_1$.
For the geometric genus, the multiplicity, and the embedding dimension of these singularities, see Table \ref{tab:special}, where the rightmost column indicates the subsections which include the details.
\begin{table}[htb]
\renewcommand{\arraystretch}{1.2}
\[
\begin{array}{ccccl}
\hline\hline
\text{type} & p_g & \mult & \emb &  \text{Section} \\
\hline
\text{Brieskorn CI} & 8 & 6 & 4 &  \text{\sref{ss:BCI2334}} \\
\hline
\text{maximal $p_g$} & 10 & 4 & 4 & \text{\sref{ss:maxpg}}  \\
\hline\hline
\end{array}
\]
\caption{\label{tab:special}
Special types}
\end{table}

Next, in \sref{ss:M=Z}, we give a complete classification of the weighted homogeneous normal surface singularities $\V$ with resolution graph $\Gamma=\Gamma(2,3,3,4)$ such that $M_X=Z_X$. 
We can see the fundamental invariants of those singularities in Table \ref{tab:M=Z}.
For each class, we prove the existence of the singularities by showing the explicit description of the  Pinkham-Demazure divisor (cf. \sref{ss:M=Z}).

\begin{table}[h]
\renewcommand{\arraystretch}{1.2}
\[
\begin{array}{ccccl}
\hline\hline
p_g & \mult & \emb & \text{ring} & \text{Proposition} \\
\hline
8 & 3 & 4 & \ \ \text{non Gorenstein} \ \  & \ref{p:h3=1} \\
 8 & 4 & 4 & \text{non Gorenstein} & \ref{p:D4=2} (1) \\
 7 & 4 & 5 & \text{non Gorenstein} & \ref{p:D4=2} (2) \\
 8 & 5 & 5 & \text{Gorenstein} & \ref{p:011} (1) \\
 7 & 5 & 5 & \text{non Gorenstein} & \ref{p:011} (2) \\
 6 & 6 & 7 & \text{non Gorenstein} & \ref{p:0101} \\
\hline\hline
\end{array}
\]
\caption{\label{tab:M=Z}
Singularities with $M_X=Z_X$}
\end{table}
Note that for any two singularities in Table \ref{tab:M=Z}, they have the same thick-thin decomposition if and only if they have the same multiplicity; see \cite{thick-thin} and the proof of \proref{p:M=Z2334} (2).

This paper is organized as follows.
In \sref{s:Pre}, we review basic facts on weighted homogeneous surface singularities and introduce the singularity of hyperelliptic type.
In \sref{s:BCI}, first we summarize the results on Brieskorn complete intersection surface singularities, and prove \thmref{t:simple} and \thmref{t:int-exist}. 
In \sref{s:exBCI},  we study weighted homogeneous singularities with resolution graph $\Gamma=\Gamma(2,3,3,4)$ such that $p_g=p_g(\Gamma)$ and those with $M_X=Z_X$. 

\begin{acknowledgement}
The author would like to thank the referee for reading the paper carefully and providing several thoughtful comments which helped improving the paper, especially, \lemref{l:Pi} and \proref{p:maxpg3}.
\end{acknowledgement}

\section{Preliminaries}\label{s:Pre}

Let $(V,o)$ be a normal complex surface singularity, namely, the germ of a normal complex surface $V$ at $o\in V$.
We denote by $\m$ the maximal ideal of the local ring $\cO_{V,o}$.
Let $\pi\: X \to V$ denote the minimal good resolution of the singularity $(V,o)$ with exceptional set $E= \pi^{-1}(p)$, 
and let $\{E_i\}_{i\in \cal I}$ denote the set of irreducible components of $E$.
We denote by $\Gamma$ the {\em resolution graph} of $\V$, namely, the weighted dual graph of $E$.
A divisor on $X$ supported in $E$ is called a {\em cycle}.
We denote the group of cycles by $\Z E$.
An element of $\Q E:=\Z E\otimes \Q$ is called a {\em $\Q$-cycle}.
Since the intersection matrix
$(E_i E_j)$ is negative definite, for every $j\in
\cal I$ there exists an effective $\Q$-cycle $E_j^*$ such that $E_j^* E_i=-\delta_{ji}$,
where $\delta_{ji}$ denotes the Kronecker delta.
Let $\Z E^*\subset \Q E$ denote the subgroup generated by $\{E_i^*\}_{i\in I}$. 

For any $\Q$-divisor $F=\sum c_iF_i$ with distinct irreducible components $F_i$, we denote by $\cf_{F_i}(F)$ the coefficient of $F_i$ in $F$, i.e., $\cf_{F_i}(F)=c_i$.
For a function $h\in H^0(\cO_X)\setminus \{0\}$, we denote by $(h)_E\in \Z E$ the {\em exceptional part} of the divisor $\di_X(h)$; 
this means that $\di_X(h)-(h)_E$ is an effective divisor containing no components of $E$. 
We call $\di_X(h)-(h)_E$ the {\em non-exceptional part} of $\di_X(h)$.
We simply write $(h)_E$ instead of $(h\circ \pi)_E$ for $h\in \m\setminus\{0\}$.

A $\Q$-cycle $D$ is said to be {\em nef} (resp. {\em anti-nef}) if $DE_i\ge 0$ (resp. $DE_i\le 0$) for all $i\in \cal I$.
Note that if a cycle $D\ne 0$ is anti-nef, then $D\ge E$.

\begin{defn}\label{d:cycles}
The {\em fundamental cycle} is by definition the smallest non-zero anti-nef cycle and denoted by $Z_X$.
The {\em maximal ideal cycle} on $X$ is the minimum of $\defset{(h)_E}{h \in \m\setminus\{0\}}$ and denoted by $M_X$. 
Clearly, $Z_X\le M_X$.
There exists a $\Q$-cycle $Z_{K_X}$ such that $(K_X+Z_{K_X})E_i=0$ for every $i\in \cI$, where $K_X$ is a canonical divisor on $X$.
We call $Z_{K_X}$ the {\em canonical cycle} on $X$. 
\end{defn}

\subsection{Cyclic quotient singularities}\label{ss:cyc}
Let $n$ and $\mu$ be positive integers with $\mu<n$
and $\gcd(n,\mu)=1$. 
 Let $\epsilon_{n}\in \C$ denote the primitive $n$-th root of unity 
and let $G$ denote the cyclic group $\left\langle \begin{pmatrix}
\epsilon_{n} & 0 \\ 0 & \epsilon_{n}^{\mu}
\end{pmatrix}  \right \rangle \subset GL(2,\C)$.
Suppose that 
$V=\C^2/G$.  Then $\V$ is called the cyclic quotient singularity of type $C_{n,\mu}$.
For integers $c_i\ge 2$, $i=1, \dots, r$, we put 
\[
[[c_{1}, \dots ,c_{r} ]]:=c_{1}-
  \cfrac{1}{c_{2}- \cfrac{1}{\ddots -\cfrac{1} 
{c_{r}}}}
\]
If $n/\mu=[[c_{1}, \dots ,c_{r} ]]$, the resolution graph $\Gamma$
 is a chain as in \figref{fig:HJ}, where all components $E_i$ are rational.

\begin{figure}[htb]
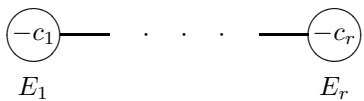

$
\xy
 (15,0)*+{-c_{1}}*\cir<10pt>{}="B"; 
(55,0)*+{-c_{r}} *\cir<10pt>{}="C";
(30,0)*+{\cdot}; 
(35,0)*+{\cdot}; 
(40,0)*+{\cdot};   
"B" *++!D(-2.0){E_1};
"C" *++!D(-2.0){E_r};
 \ar @{-}"B";(25,0)
  \ar @{-} (45,0);  "C"
\endxy
$
\caption{The resolution graph of $C_{n,\mu}$ \label{fig:HJ}}
\end{figure}
It is known that the local class group $\Cl\V$ 
is isomorphic to the finite abelian group
\[
\Z E^* /\Z E =\gen{[E_1^*]}=\gen{[E_r^*]}
\]
of order $n$, where $[E_i^*]=E_i^*+\Z E$ (cf. \cite[II (a)]{mum.top}, \cite[III. 5]{CCS}). 

Suppose that $E_0$ is a prime divisor on $X$ such that $E_0E_i=\delta_{1 i}$ for $1\le i \le r$; so $E_0+E_1+\cdots+E_r$ looks like a chain of curves.
For any positive integer $m_0$, let 
\[
\cal L(m_0 )=\defset{m_0 E_0+\sum_{i=1}^r m_iE_i}{m_1, \dots, m_r\in \Z_{> 0}}.
\]
Then we define a set $\cal D(m_0)$ as follows:
\[
\cal D(m_0 ):=\defset{D\in \cal L(m_0)}{DE_i\le 0, \; i=1,\dots, r}.
\]
It is easy to see that 
$\cal D(m_0 )$ is not empty and has a unique smallest element.

Let $\Ce{x}$ denote the ceiling of a real number $x$.

\begin{lem}\label{l:La}
Let $D\in \cal D(m_0 )$.
We have the following:
\begin{enumerate}
\item There exists an effective cycle $F$ such that $(D+F)E_i= 0$ for $1\le i<r$ and $\supp(F)\subset \bigcup_{i>1}E_i$.
\item If $DE_i= 0$ for $1\le i<r$ and $DE_r\ge -1$, then $D$ is the smallest element of $\cal D(m_0 )$.
\item Assume that $D, D'\in \cal D(m_0 )$ 
and $DE_i=D'E_i$ for $1\le i<r$.
If $D>D'$, then $\cf_{E_1}(D)>\cf_{E_1}(D')$.
\item Assume that $D$ and  $D'$ are the smallest elements of $\cal D(m_0)$ and $\cal D(m_0')$, respectively,
and that $D'E_i=0$ for $1\le i \le r$. 
Then $D+D'$ is the smallest element of $\cal D(m_0+m_0')$.
\end{enumerate}
\end{lem}
\begin{proof}
We write as  $D=\sum_{i=0}^r m_iE_i$ and 
$D'=\sum_{i=0}^r m_i'E_i$.

(1) For any $1\le k < r$, there exists a cycle $F'$ supported on $E_{k+1}+\cdots+E_r$ such that 
\[
\cf_{E_{k+1}}(F')=1,\ \ F'E_{k+1}=\cdots=F'E_{r-1}=0, \quad F'E_r < 0
\]
(cf. \cite[III.5]{CCS}).
If $a:=DE_k<0$, then $D+aF'\in \cal D(m_0)$ and  $(D+aF')E_{k}=0$. By repeating this process, 
we obtain the assertion.

(2) It follows from \cite[Lemma 2.2]{la.TanCon} (cf. \cite[2.1]{MO}).

(3) If $m_1=m_1'$, we can take $1\le k<r$ so that $m_i=m_i'$ for $i\le k$ and $m_{k+1}>m_{k+1}'$. Then $(D-D')E_k=m_{k+1}-m_{k+1}'>0$; it contradicts that $DE_k=D'E_k$.

(4)  Let $d_i=[[c_i, \dots, c_r]]$. 
By \cite[Lemma 1.1]{K-N}, the minimality of $D$ is characterized by the condition that $m_i=\Ce{m_{i-1}/d_i}$ for $1\le i \le r$.
By the assumption, it follows from Lemma 1.2 (1) and (2) of \cite{K-N} that 
$m_i'=m_{i-1}'/d_i$.
Hence we have $m_i+m_i'=\Ce{m_{i-1}/d_i}+m_{i-1}'/d_i=\Ce{(m_{i-1}+m_{i-1}')/d_i}$. 
\end{proof}

\subsection{Weighted homogeneous surface singularities}
\label{ss:WH}
Let us recall some fundamental facts on weighted homogeneous surface singularities (cf. \cite{p.qh}).

Assume that $\V$ is a weighted homogeneous singularity.
Then the resolution graph $\Gamma$ of $\V$ is a star-shaped graph as in \figref{fig:star}, where $E_{i,j}$ are rational curves, $g$ is the genus of the curve $E_0$, $c_{i,j}$ and $c_{0}$ are the self-intersection numbers.
The component $E_0$ is called the {\em central curve}.

\begin{figure}[htb]
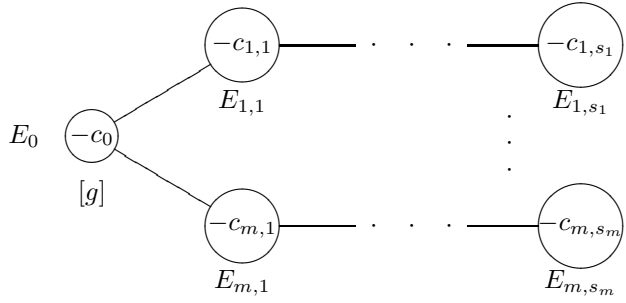

 \begin{center}
$
\xy
(-9,0)*+{E_0}; (0,0)*+{-c_0}*\cir<10pt>{}="E"*++!D(-2.0){[g]}; 
(20,12)*+{-c_{1,1} }*\cir<14pt>{}="B_1"*++!D(-2.0){E_{1,1}}; 
(65,12)*+{-c_{1,s_1} }*\cir<16pt>{}="B_3"*++!D(-2.0){E_{1,s_1}}; 
(20,-12)*+{-c_{m,1} }*\cir<14pt>{}="D_1"*++!D(-2.0){E_{m,1}}; 
(65,-12)*+{-c_{m,s_{m}} }*\cir<16pt>{}="D_3"*++!D(-2.0){E_{m,s_{m}}};
(37.5,12)*{\cdot },(42.5,12)*{\cdot },(47.5,12)*{\cdot }, 
(37.5,-12)*{\cdot},(42.5,-12)*{\cdot},(47.5,-12)*{\cdot}, 
(55.5,2.5)*{\cdot},(55.5,-1)*{\cdot},(55.5,-4.5)*{\cdot}, 

\ar @{-} "E" ;"B_1" 
\ar @{-} "E" ;"D_1"  

\ar @{-}"B_1";(35,12) \ar @{-} (50,12);"B_3" 
\ar @{-} "D_1";(35,-12) \ar @{-} (50,-12);"D_3" 
\endxy
$
\end{center}
\caption{\label{fig:star} A star-shaped resolution graph}
\end{figure}
For $1\le i \le m$, we define positive integers $\alpha_i$ and $\beta_i$ with $\gcd(\alpha_i, \beta_i)=1$ by $\alpha_i/\beta_i=[[c_{i,1}, \dots ,c_{i,s_i} ]]$. 
The data 
\[
(g, c_0, (\alpha_1, \beta_1), \dots, (\alpha_m,\beta_m))
\]
 is called the {\em Seifert invariant}.
Note that the graph $\Gamma$ can be recovered from the Seifert invariant.

Let $P_i\in E_0$ denote the point $E_0\cap E_{i,1}$ and $Q$ a divisor on $E_0$ such that $\cO_{E_0}(-E_0)\cong \cO_{E_0}(Q)$. 
We define a $\Q$-divisor $D$ and divisors $D_k$ ($k\in \Z_{\ge 0}$) on $E_0$ as follows:
\[
D:=Q-\sum_{i=1}^m\frac{\beta_i}{\alpha_i}P_i, \qquad
D_k:=kQ-\sum_{i=1}^m \Ce{\frac{k\beta_i}{\alpha_i}} P_i.
\]
We call $D$ the {\em Pinkham-Demazure divisor}.
It is known that $\deg D >0$.
For any divisor $F$ on $E_0$, we write as 
\[
H^i(F)=H^i(\cO_{E_0}(F)), \quad h^i(F)=\dim_{\C}H^i(F).
\]

Let $R:=R\V$ denote the homogeneous coordinate ring of the singularity $(V,o)$.
Then we have the expression  $R=\bigoplus_{k\ge 0}H^0(D_k)T^k\subset \C(E_0)[T]$, where $\C(E_0)$ is the field of rational functions on $E_0$ and $T$ an indeterminate (cf. \cite{p.qh}, \cite{tki-w}).
We have the following.

\begin{thm}[Pinkham \cite{p.qh}]
\label{t:Pin}
$p_g\V=\sum_{k\ge 0}h^1(D_k)$.
\end{thm}

Let $H(V,t)$ denote the Hilbert series of the graded ring $R$, i.e., $H(V,t)=\sum_{k\ge 0}h^0(D_k)t^k$. 

\begin{prop}\label{p:Hpg}
We have the following.
\begin{enumerate}
\item If we write as $H(V,t)=p(t)/q(t)+r(t)$, where $p,\, q, \, r\in\C[t]$ and $\deg p<\deg q$, then $p_g\V=r(1)$. 
\item Let $(V_1, o_1)$ and $(V_2,o_2)$ be weighted homogeneous singularities with the same resolution graph. Then $p_g(V_1,o_1)-p_g(V_2,o_2)=(H(V_1,t)-H(V_2,t))|_{t=1}$.
\end{enumerate}
\end{prop}
\begin{proof}
(1) follows from  \cite[3.1.3]{no-edwh}.

(2) 
It follows from \thmref{t:Pin} and the Riemann-Roch theorem $h^0(D_n)-h^1(D_n)= \deg D_n+1-g$ (the right-hand side is determined by $\Gamma$).
\end{proof}

The next theorem follows from \cite[2.9]{KeiWat-D}.

\begin{thm}\label{t:WatD}
Let $D'=\sum  ((\alpha_i-1)/\alpha_i)P_i$.
Then $R$ is Gorenstein if and only if there exists an integer $a$ such that $K_C\sim aD-D'$;  the integer $a$ coincides with the $a$-invariant $a(R)$ of Goto--Watanabe  (\cite{G-W}).
\end{thm}

\subsection{Surface singularities with star-shaped graph}
\label{ss:star}

First, we briefly review some important facts in \cite[\S 6]{tki-w}.
Assume that $\V$ is a normal surface singularity with star-shaped resolution graph $\Gamma$ as  \figref{fig:star}.
Even if $\V$ is not weighted homogeneous, in the same manner as in \sref{ss:WH}, we obtain the Pinkham-Demazure divisor
\[
D=Q-\sum_{i=1}^m\frac{\beta_i}{\alpha_i}P_i
\]
on the central curve $E_0\subset E$ on the minimal good resolution $X$.
We define the graded ring $R$ by
\[
R=R(E_0, D):=\bigoplus_{k\ge 0}H^0(D_k)T^k\subset \C(E_0)[T].
\]
Let $\bV=\spec R$ and $o\in \bV$ the point defined by the maximal ideal $\bigoplus_{k\ge 1}H^0(D_k)T^k$.
Then $(\bV,o)$ is a weighted homogeneous normal surface singularity with resolution graph $\Gamma$.

\begin{thm}[Tomari-Watanabe {\cite[\S 6]{tki-w}}]\label{t:TW}
For every $n\in\Z_{\ge 0}$, there exists the minimal cycle $L_n\in \Z E$ such that $L_n$ is anti-nef on $E-E_0$ and $\cf_{E_0}(L_n)=n$.\footnote{Our symbol $L_n$ is equal to $-L_{-n}$ in \cite[\S 6]{tki-w}.}
Then we have a natural isomorphism $\cO_{E_0}(-L_n)\cong \cO_{E_0}(D_n)$ for $n\in\Z_{\ge 0}$; in fact, 
\[
\sum_{i=1}^m \Ce{\frac{k\beta_i}{\alpha_i}} P_i=(L_n-nE_0)|_{E_0}.
\]
In general, we have $p_g\V \le p_g(\bV,o)$.
If the equality $p_g\V =p_g(\bV,o)$ holds, the following sequence is exact for $n\ge 0$:
\[
0\to H^0(\cO_X(-L_n-E_0)) \to H^0(\cO_X(-L_n)) \to H^0(\cO_{E_0}(D_n)) \to 0.
\]
\end{thm} 

\begin{rem} \label{r:E0coeff}
From the definitions of $Z_X$ and $M_X$, 
we have the following: 
\begin{align*}
\cf_{E_0}(Z_X)& =\min\defset{m\in \Z_{>0}}{\deg D_m\ge 0}, \\
\cf_{E_0}(M_X)&=\min\defset{m\in \Z_{>0}}{H^0(D_m)\ne 0}.
\end{align*}
Clearly, $z_0:=\cf_{E_0}(Z_X) \le m_0:=\cf_{E_0}(M_X)$. 
One of fundamental problems is to find a characterization for the equality $z_0=m_0$. 
We have $Z_X=L_{z_0}$ by the definition of the cycles $L_n$.
It might be natural to ask  whether the condition $m_0=z_0$ implies the equality $M_X=Z_X$.
For Brieskorn complete intersection singularities, we have a criterion for $z_0=m_0$ and we always have $M_X=L_{m_0}$ (see \cite{K-N}, \cite{MO}).
However, in general, this is not true even for weighted homogeneous singularities (see \cite{TT}).
We will see later (\proref{p:maxpg3}) an example of a weighted homogeneous singularity homeomorphic to a Brieskorn complete intersection singularity which does not satisfy $M_X=L_{m_0}$ though $z_0=m_0$ and has the ``maximal geometric genus'' in the following sense.
\end{rem}

\begin{defn}\label{d:pgG}
Let $\cal X(\Gamma)$ denote the set of normal surface singularities with resolution graph $\Gamma$ and let
\[
p_g(\Gamma):=\max \defset{p_g(W,o)}{(W,o)\in \cal X(\Gamma)}.
\]
\end{defn}

Obviously, $p_g(\Gamma)$ is an invariant of $\Gamma$.
From \thmref{t:TW},  $p_g(\Gamma)$ 
is attained by a weighted homogeneous singularity.
However, the inequality $p_g(\bV,o)<p_g(\Gamma)$ may happen in general, namely, $p_g(\bV,o)$ is not topological, even if $\Gamma$ is a resolution graph of a Brieskorn complete intersection singularity (see \sref{s:exBCI}).

Let $\Fl{x}$ denote the floor (or, integer part) of a real number $x$.

\begin{ex}\label{e:mpg}
Assume that a resolution graph $\Gamma$ has the Seifert invariant 
\[
(g, c_0, k_1(\alpha_1, \beta_1), \dots, k_m(\alpha_m,\beta_m)),
\]
 where $k_i(\alpha_i,\beta_i)$ means that $(\alpha_i,\beta_i)$ is repeated $k_i$ times, and $(\alpha_i,\beta_i)\ne (\alpha_j,\beta_j)$ for $i\ne j$.
Moreover, assume that $k_2, \dots, k_m\in 2\Z$; in this case, we call $\Gamma$ a {\em hyperelliptic type}.

Let $C$ be a hyperelliptic or elliptic curve of genus $g$ and let $\cal R(C)$ be the set of ramification points of the double cover $C\to \P^1$ with 
involution $\sigma\: C\to C$.
Let $P\in \cal R(C)$ and $Q=c_0P$. Take $P_{i,j}\in C\setminus\cal R(C)$ ($1\le i \le m$, $1\le j \le \Fl{k_i/2}$) so that $P_{1,1}, \sigma(P_{1,1}), \dots, P_{m, \Fl{k_m/2}}, \sigma(P_{m, \Fl{k_m/2}})$ are different from each other.
Let $Q_{i,j}=P_{i,j}+\sigma(P_{i,j})$. 
Then we define  the Pinkham-Demazure divisor $D$ on $C$ by
\[
D=\begin{cases}
\dis Q-\sum_{i=1}^m \frac{\beta_i}{\alpha_i} \sum_{j=1}^{k_i/2}Q_{i,j}
& \text{ if } k_1 \in 2\Z, \\
\dis Q-\frac{\beta_1}{\alpha_1}P
-\frac{\beta_1}{\alpha_1}\sum_{j=1}^{(k_1-1)/2}Q_{1,j}-\sum_{i=2}^m \frac{\beta_i}{\alpha_i} \sum_{j=1}^{k_i/2}Q_{i,j} & \text{ if } k_1 \not\in 2\Z.
\end{cases}
\]
Since $Q_{i,j}\sim 2P$, we have $D_n\sim (\deg D_n)P$.
Let $R=\bigoplus_{k\ge 0}H^0(D_k)T^k$ and $\bV=\spec R$.
We say that the weighted homogeneous normal surface singularity $(\bV,o)$ is a {\em hyperelliptic type}, too.
Then the singularity $(\bV,o)$ has the resolution graph $\Gamma$ and 
$p_g(\bV,o)=p_g(\Gamma)$, because it follows from Clifford's theorem that $h^1(D_n)$ is the maximum of $h^1(D'_n)$, where $C'$ is any nonsingular curve of genus $g$ and 
$D'$ is any Pinkham-Demazure divisor on $C'$ which corresponding to the resolution graph $\Gamma$.
\end{ex}

The following problems are open even for Brieskorn complete intersections.

\begin{prob}
Give an explicit way to compute $p_g(\Gamma)$ from $\Gamma$.
\end{prob}

\begin{prob}
Classify complex structures which attain $p_g(\Gamma)$.
Is $E_0$ always hyperelliptic if $p_g\V=p_g(\Gamma)$?
\end{prob}

\begin{prob}
How can we generalize the notion of ``hyperelliptic type'' to non-star-shaped cases?
\end{prob}

\section{Brieskorn complete intersection singularities}
\label{s:BCI}
In this section, we review some basic facts on the Brieskorn complete intersection (BCI for short) surface singularities and study arithmetic properties of invariants of those singularities.
Then we show that a BCI singularity with $g\le 1$ always has the maximal geometric genus and its maximal ideal cycle coincides with the fundamental cycle on the minimal good resolution.
We basically use the notation of \sref{s:Pre}.

Recall that $\pi\: X\to V$ denotes the minimal good resolution of a normal surface singularity $\V$ with exceptional set $E$.

\subsection{The cycles and the Seifert invariants}\label{ss:BCISf}

We summarize the results in \cite{MO} which will be used in this section; those are a natural extension of the results on the hypersurface case obtained by Konno and Nagashima \cite{K-N}.
We assume that $\V$ is a BCI normal surface singularity, namely, 
$V\subset \C^m$ can be defined as
\begin{equation}\label{eq:BCIeq}
V=\defset{(x_i)\in \C^m}
{q_{i1}x_1^{a_1}+\cdots +q_{im}x_{m}^{a_{m}} =0,
\quad i=3,\dots , m}, 
\end{equation}
where $a_i$ are integers such that $2\le a_1\le \dots \le a_m$ 
and $q_{ij}\in \C$. 

We define  positive integers 
$\ell$, $\ell_i$, $\alpha$, $\alpha_i$, $\beta_i$, 
$\hat g$, $\hat g_i$, and $e_i$ as follows:\footnote{Using the notation of \cite[\S 3]{MO}, we have $l=d_m$, $\ell_i=d_{im}$, $\alpha_i=n_{im}$, $\beta_i=\mu_{im}$, $e_i=e_{im}$. }
\begin{gather*}
\ell:=\lcm\{a_1, \dots, a_m\}, \ \ 
\ell_i:=\lcm(\{a_1, \dots, a_m\}\setminus\{a_i\}),  \\ 
\alpha_i:=\ell/\ell_i,  \ \
\alpha:=\alpha_1\cdots \alpha_m, \ \ 
\hat g:=a_1\cdots a_{m}/\ell, \ \ 
\hat g_i:=\hat g \alpha_i/a_i, \ \ 
e_i:=\ell/a_i, \\
e_i\beta_i+1 \equiv 0 \pmod{\alpha_i} \; \text{ and } \; 0\le \beta_i<\alpha_i. 
\end{gather*}
We easily see that the polynomials appeared in \eqref{eq:BCIeq} are weighted homogeneous polynomials of degree $\ell$ with respect to the weights $(e_1, \dots, e_m)$ and that $\gcd\{\alpha_i ,\alpha_j\}=1$ for $i\ne j$.

\begin{defn}
Let $Z^{(i)}=(x_i)_E$, the exceptional part of the divisor $\di_X(x_i)$.
\end{defn}

The next result follows from Theorem 4.4, 5.1, 6.1 of \cite{MO}.

\begin{thm}\label{t:BCImain}
We have the following.
\begin{enumerate}
\item The resolution graph of $\V$  is as in \figref{fig:BCIG} ($s_i=0$ if $\alpha_i=1$), where 
\[
E=E_0+\sum_{i=1}^{m}\sum_{\nu=1}^{s_{i}}
\sum_{\xi=1}^{\hat g_i}E_{i,\nu,\xi}, 
\]
and the Seifert invariant is given by the following:
\begin{gather*}
2g-2=(m-2)\hat g -\sum_{i=1}^{m}\hat g_i, 
 \\ 
c_0=\sum _{i=1}^{m}
\frac{\hat g_i\beta_i}{\alpha_i}
+\frac{a_1\cdots a_{m}}{\ell^2}, \ \ 
\beta_i/\alpha_i=
\begin{cases}
[[c_{i,1}, \dots, c_{i,s_i}]]^{-1} & \text{if} \ \  \alpha_i\ge 2 \\
0 &  \text{if} \ \ \alpha_i=1.
\end{cases}
\end{gather*}
\item For $1\le i \le m$, we have   
\[
\cf_{E_0}(Z^{(i)})=e_i=\deg(x_i), \quad 
Z^{(i)}=\begin{cases}
\sum_{\xi=1}^{\hat g_i}E_{i,s_i,\xi}^* 
& \text{if} \ \  \alpha_i\ge 2 \\
\hat g_i E_0^* &  \text{if} \ \ \alpha_i=1.
\end{cases}
\]
Hence $Z^{(i)}=L_{e_i}$ for $1\le i \le m$, 
 and 
$M_X=Z^{(m)}$ since $e_1 \ge \cdots \ge e_m$.

\item We have $\cf_{E_0}(Z_X)=\min\{e_m, \alpha\}$ (cf. \remref{r:E0coeff}) 
and 
\[
Z_X=\begin{cases} M_X & \text{if} \ \  e_m\le \alpha \\
\deg (\alpha D) E_0^*  & \text{if} \ \  e_m> \alpha.
\end{cases}
\]
In particular, $Z_X=M_X$ if and only if $e_{m}\le \alpha$.
\end{enumerate}
\end{thm}

\begin{figure}[htb]
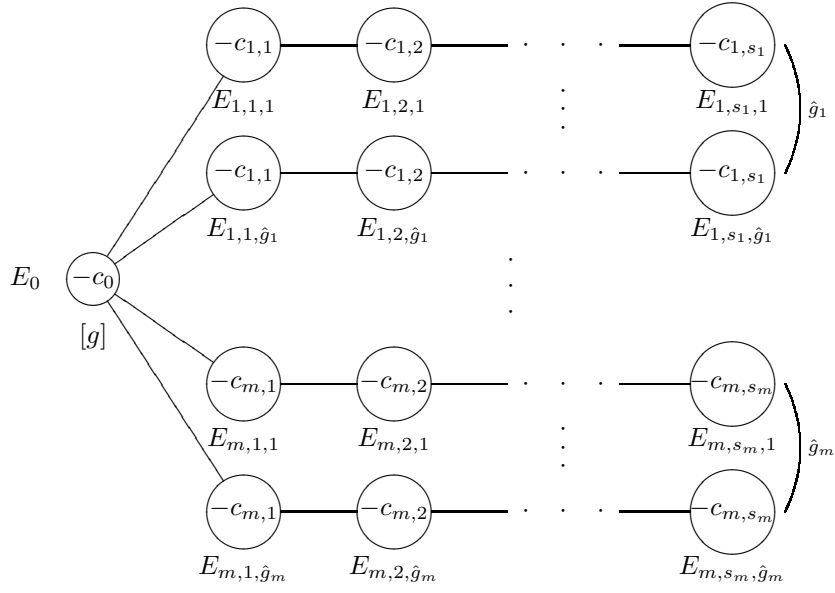

 \begin{center}
$
\xy
(-9,0)*+{E_0}; (0,0)*+{-c_0}*\cir<10pt>{}="E"*++!D(-2.0){[g]}; 
(20,31)*+{-c_{1,1} }*\cir<14pt>{}="A_1"*++!D(-2.0){E_{1,1,1}}; 
(40,31)*+{-c_{1,2} }*\cir<14pt>{}="A_2"*++!D(-2.0){E_{1,2,1}}; 
(85,31)*+{-c_{1,s_1} }*\cir<16pt>{}="A_3"*++!D(-2.0){E_{1,s_1,1}}; 
(20,14)*+{-c_{1,1} }*\cir<14pt>{}="B_1"*++!D(-2.0){E_{1,1,\hat g_1}}; (40,14)*+{-c_{1,2} }*\cir<14pt>{}="B_2"*++!D(-2.0){E_{1,2,\hat g_1}}; (85,14)*+{-c_{1,s_1} }*\cir<16pt>{}="B_3"*++!D(-2.0){E_{1,s_1,\hat g_1}}; 
(20,-14)*+{-c_{m,1} }*\cir<14pt>{}="D_1"*++!D(-2.0){E_{m,1,1}}; 
(40,-14)*+{-c_{m,2} }*\cir<14pt>{}="D_2"*++!D(-2.0){E_{m,2,1}}; 
(85,-14)*+{-c_{m,s_{m}} }*\cir<16pt>{}="D_3"*++!D(-2.0){E_{m,s_{m},1}};
(20,-31)*+{-c_{m,1} }*\cir<14pt>{}="E_1"*++!D(-2.0){E_{m,1,\hat g_{m}}};
(40,-31)*+{-c_{m,2} }*\cir<14pt>{}="E_2"*++!D(-2.0){E_{m,2,\hat g_{m}}};
(85,-31)*+{-c_{m,s_{m}} }*\cir<16pt>{}="E_3"*++!D(-2.0){E_{m,s_{m},\hat g_{m}}}; 
(57.5,31)*{\cdot },(62.5,31)*{\cdot },(67.5,31)*{\cdot }, 
(57.5,14)*{\cdot },(62.5,14)*{\cdot },(67.5,14)*{\cdot }, 
(57.5,-14)*{\cdot},(62.5,-14)*{\cdot},(67.5,-14)*{\cdot}, 
(57.5,-31)*{\cdot},(62.5,-31)*{\cdot},(67.5,-31)*{\cdot}, 
(62.5,25)*{\cdot},(62.5,22.5)*{\cdot},(62.5,20)*{\cdot},    
(55.5,2.5)*{\cdot},(55.5,-1)*{\cdot},(55.5,-4.5)*{\cdot}, 
(62.5,-20)*{\cdot},(62.5,-22.5)*{\cdot},(62.5,-25)*{\cdot}, 

\ar @{-} "E" ;"A_1"  
\ar @{-} "E" ;"B_1" 
\ar @{-} "E" ;"D_1"  
\ar @{-} "E" ;"E_1"

\ar @{-} "A_1"; "A_2"   
\ar @{-} "A_2"; (55,31) \ar @{-} (70,31);"A_3"  

\ar @{-} "B_1"; "B_2" 
\ar @{-}"B_2";(55,14) \ar @{-} (70,14);"B_3" 
\ar @/^2mm/@{-}^{\hat g_1} (92,31);(92,14) 

\ar @{-} "D_1" ; "D_2"  
\ar @{-} "D_2";(55,-14) \ar @{-} (70,-14);"D_3" 

\ar @{-} "E_1" ; "E_2" 
\ar @{-} "E_2";(55,-31) \ar @{-} (70,-31);"E_3" 
\ar @/^2mm/@{-}^{\hat g_{m}} (92,-14);(92,-31) 
\endxy
$
\end{center}
\caption{\label{fig:BCIG} The graph $\Gamma(a_1, \dots, a_m)$}
\end{figure}

\begin{defn}
We denote the weighted dual graph of \figref{fig:BCIG} by $\Gamma(a_1, \dots, a_m)$.
\end{defn}

\begin{rem}\label{r:ZiHi}
We describe more precisely the situation of \thmref{t:BCImain} (2).
Let $H_i:=\di_X(x_i)-Z^{(i)}$. Then we have the decomposition $H_i=\bigcup_{\xi=1}^{\hat g_i} H_{i,\xi}$ into irreducible components such that 
\begin{itemize}
\item $H_{i,\xi} E=H_{i,\xi} E_{i,s_i,\xi}=1$ if $\alpha_i\ne 1$, 
\item $H_{i,\xi} E=H_{i,\xi} E_{0}=1$ and $H_{i,\xi}\cap H_{i,\xi'}=\emptyset$ ($\xi\ne \xi'$) if $\alpha_i=1$.
\end{itemize}
In any cases, $H_i\cap H_j=\emptyset$ for $i\ne j$.

For $1\le i \le m$, let 
$\defset{P_{i\xi}}{\xi=1, \dots, \hg_i}\subset E_0$
 denote the set of points determined by $x_i=0$ in the weighted projective space $\P(e_1, \dots, e_m)$.
Then 
\[
\{P_{i\xi}\}=
\begin{cases}
E_0\cap E_{i,1,\xi} & \text{if $\alpha_i\ne 1$,} \\
E_0\cap H_{i,\xi} & \text{if $\alpha_i= 1$}.
\end{cases}
\]
\end{rem}

Let us recall that $\cO_{E_0}(-L_n)\cong \cO_{E_0}(D_n)$ (see \thmref{t:TW}) and $D_{\alpha}=\alpha D$.

\begin{lem}\label{l:E_0*}
We have the following.
\begin{enumerate}
\item For $n\in \Z_{> 0}$, $\alpha \mid n$  if and only if 
$L_n=(\deg D_n)E_0^*$.
In particular, if $\deg D_{e_i}>0$, then $\alpha \mid e_i$.
\item If $d\in \Z_{>0}$ and $dE_0^*\in \Z E$,  then $dE_0^*=L_n$, where $n=d\alpha/\deg D_{\alpha}$.
\end{enumerate}

\end{lem}
\begin{proof}
(1)
Let $\phi\: X\to X'$ be the blowing-down of the divisor $E-E_0$. 
Then, at each point $\phi(P_{i\xi})\in X'$ ($1\le i\le m$, $1\le \xi\le \hat g _i$), the reduced divisor $\phi(E_0)$ is a $\Q$-Cartier divisor and the order of $[\phi(E_0)]\in \Cl(X', \phi(P_{i\xi}))$ is $\alpha_i$  (see \sref{ss:cyc}).
As in \cite[II (b)]{mum.top}, we have the pull-back $\phi^*\phi(E_0)$. Then $E_0^*=\cf_{E_0}(E_0^*)(\phi^*\phi(E_0))$.
Since $\alpha_i$'s are pairwise relatively prime, 
$\alpha$ is the minimal positive integer such that $\alpha\phi(E_0)$ is a Cartier divisor on $X'$, 
or equivalently, $\phi^*(\alpha\phi(E_0))\in \Z E$.
Hence $\alpha \mid n$ if and only if $\phi^*(n\phi(E_0))\in \Z E$.
If this is the case, $\phi^*(n\phi(E_0))=L_{n}$ by \lemref{l:La} (2), and moreover, $L_{n}=(-L_nE_0)E_0^*=(\deg D_{n})E_0^*$.
By \thmref{t:BCImain} (2), $L_{e_i}=(\deg D_{e_i})E_0^*$ if $\deg D_{e_i}>0$. 

(2)
As seen above, $dE_0^*=L_n$ by \lemref{l:La} (2).
Then $n=d\cf_{E_0}(E_0^*)$.
From (1), we have $\alpha=\deg D_{\alpha}\cf_{E_0}(E_0^*)$.
\end{proof}

\subsection{The coordinate ring and the semigroups}

By virtue of \thmref{t:BCImain}, we can write down the Pinkham-Demazure divisor as follows:
\[
D=Q-\Delta, \quad \Delta=\sum_{i=1}^m \frac{\beta_i}{\alpha_i}\bar P_i, \quad \bar P_i=\sum_{\xi=1}^{\hg_i}P_{i\xi}
\quad \text{($\beta_i=0$ if $\alpha_i=1$)}.
\]

\begin{defn}
We call a cycle $C\ge 0$ a {\em monomial cycle} if 
 $C=\sum_{i=1}^mm_iZ^{(i)}$ with $m_i\in \Z_{\ge 0}$, and write $x(C)=\prod_{i=1}^mx_i^{m_i}$.
Clearly, $(x(C))_E=C$.
\end{defn}

\begin{rem}\label{r:monom}
Let $C>0$ be an anti-nef $\Q$-cycle.
Suppose that $\alpha_i>1$ for $i\le s$ and $\alpha_i=1$ for $i>s$.
If, for each $i\le s$, $c_i:=CE_{i,s_i,\xi}$ is non-negative integer independent of $1\le \xi \le \hat g_i$,  and if the intersection numbers of $C$ and the exceptional components other than $E_{i,s_i,\xi}$ ($i\le s$, $1\le \xi\le \hat g_i$) are zero, then $C$ is a monomial cycle since $C=\sum_{i=1}^sc_iZ^{(i)}$.

On the other hand, even if $C\in \Z E$ and $C=cE_0^*$ for some $c\in \Z_{>0}$, $C$ is not necessarily a monomial cycle.
For example, if $\alpha<e_m$, then $L_{\alpha}=(\deg D_{\alpha})E_0^*$  is not a monomial cycle (see \lemref{l:E_0*}, \thmref{t:BCImain} (2)).
\end{rem}

Let $\gen{m_1,\dots, m_k}\subset \Z_{\ge 0}$ denote the numerical semigroup generated by integers  $m_1, \dots, m_k\in \Z_{\ge 0}$.
For $n\in \Z_{\ge 0}$, let $R_n=H^0(D_n)T^n\subset R:=R(V,o)$, the vector space of homogeneous functions of degree $n$ (see \sref{ss:WH}). 

\begin{prop}\label{p:monomials}
Let $n\in \Z_{\ge 0}$.
We have the following.
\begin{enumerate}
\item If $\deg D_n\in \gen{\hg_1, \dots, \hg_m}$, then 
there exists a monomial cycle $W$ such that $\cf_{E_0}(W)=n$, and hence $h^0(D_n)\ne 0$.

\item If $\deg D_n=\deg D_k \in \gen{\hg_1, \dots, \hg_m}$ for some $k\ge 0$, 
then $D_n\sim D_k$.
In particular,  if $\deg D_n=0$, then $D_n\sim 0$.
\item If $d:=\deg D_n>0$, then $dE_0^*\in \Z E$ and $\deg D_{\alpha}\mid d$.
\end{enumerate}
\end{prop}
\begin{proof}
(1) 
We first assume that $\deg D_n=0$.
If $\alpha_i>1$, then  $\cf_{E_{i,j,\xi}}(L_n)$ is independent of $1\le \xi \le \hat g_i$ for each $1\le j \le s_i$ (see \figref{fig:BCIG}).
Therefore, by \lemref{l:La} (1), there exists a cycle $F>0$ such that $L:=L_n+F$ is a monomial cycle with  $\cf_{E_0}(L)=\cf_{E_0}(L_n)=n$ and $LE_0=0$ (cf. \remref{r:monom}).
Then $x(L)\in R_n$.

Next assume that $\deg D_n=c_1\hg_1+\cdots+c_m\hg_m>0$ ($c_i\in \Z_{\ge 0}$).
We may assume that $\alpha_i>1$ for $i\le s$ and $\alpha_i=1$ for $i>s$. For $i\le s$,  let $F_i=\sum_{\xi=1}^{\hat g_i}\sum_{j=1}^{s_i} E_{i,j,\xi}$. 
Since $F_i$ is anti-nef on its support and $\deg D_n = -L_n E_0$, 
it follows from \thmref{t:BCImain} (2) that the cycle
\[
W'=L_n+\sum_{i=1}^s c_i F_i-\sum_{i=s+1}^mc_iZ^{(i)}
\]
is anti-nef and  $W'E_0=0$.
Applying the argument above to the cycle $W'$, there exists a cycle $F'>0$ such that $W'+F'$ is a monomial cycle with  $\cf_{E_0}(W')=\cf_{E_0}(W'+F')$ and $(W'+F')E_0=0$.
Hence 
\[
W:=W'+F'+\sum_{i=s+1}^mc_iZ^{(i)}
\]
 is also a monomial cycle and $\cf_{E_0}(W)=\cf_{E_0}(W'+\sum_{i=s+1}^mc_iZ^{(i)})=n$. 
Thus, we obtain that $x(W)\in R_n$.

(2) 
We denote by $C_n$ the monomial cycle $W'+F'$ above, and also by $C_k$  the monomial cycle obtained from $L_k$ in the same manner as above.
Since $C_n-C_k=L_n-L_k$, on a suitably small neighborhood of $E_0\subset X$, we have
\[
L_n-L_k=\di_X(x(C_n)/x(C_k))\sim 0.
\]
Hence $D_n-D_k\sim (-L_n+L_k)|_{E_0}\sim 0$.

(3) Since $\deg D_n=-L_nE_0$, $L_n-dE_0^*$ is an anti-nef $\Q$-cycle with $(L_n-dE_0)E_0=0$.
By the argument above, there exists a cycle $F>0$ such that $L_n-dE_0^*+F$ is a monomial cycle.
Hence $dE_0^*$ is also a cycle (cf. \remref{r:monom}).
We have $\deg D_{\alpha} \mid d$ by \lemref{l:E_0*}.
\end{proof}

\begin{thm}\label{t:g1p_gmax}
If $g\le 1$, then $p_g(V,o)=p_g(\Gamma(a_1, \dots, a_m))$  (see \defref{d:pgG}). 
\end{thm}
\begin{proof}
By Pinkham's formula, $p_g(V,o)=\sum_{n\ge 0}h^1(D_n)$.
If $g=0$, then this is topological, and the assertion is clear.
Suppose that $g=1$. If $\deg D_n\ne 0$, then $h^1(D_n)$ is topological by Riemann-Roch theorem and Serre duality, namely, independent of the complex structure of $\V$.
If $\deg D_n=0$, then $h^1(D_n)=h^0(D_n)=1$ by \proref{p:monomials}.
Hence $p_g(V,o)= p_g(\Gamma(a_1, \dots, a_m))$.
\end{proof}

\begin{thm}\label{t:ehg}
We have the following.
\begin{enumerate}
\item $\gen{e_1, \dots, e_m}=\defset{n\in \Z_{\ge 0}}{h^0(D_n)\ne 0}$.
\item For $n\in \Z_{\ge 0}$, $n\in \gen{e_1, \dots, e_m}$ if and only if $\deg D_n\in \gen{\hat g_1, \dots, \hat g_m}$.
\end{enumerate}
\end{thm}
\begin{proof}
(1) follows from the fact that $R=\bigoplus_{k\ge 0}H^0(D_k)T^k$ is generated by the elements $x_1, \dots, x_m$ with $\deg x_i=e_i$.

(2) The ``if'' part follows from  \proref{p:monomials} (1).
Assume that $n=\sum_{i=1}^mm_ie_i$ with $m_i \ge 0$.
Then the monomial cycle $M:=\sum_{i=1}^mm_iZ^{(i)}$ satisfies $\cf_{E_0}(M)=n$. 
We proceed in a similar way as in the proof of \proref{p:monomials}.
We may assume that $\alpha_i>1$ for $i\le s$ and $\alpha_i=1$ for $i>s$. Then $-ME_0 =\sum_{i>s}m_i\hat g_i\in \gen{\hat g_1, \dots, \hat g_m}$ by \thmref{t:BCImain} (2).
 Let $W=M-\sum_{i>s}m_iZ^{(i)}$ and $n'=\cf_{E_0}(W)$.
Clearly, $W$ is also a monomial cycle.
By the definition of $L_{n'}$, we have  $\cf_{E_0}(W-L_{n'})=0$ and $W-L_{n'}\ge 0$.
Since  $\cf_{E_{i,j,\xi}}(L_{n'})$ and $\cf_{E_{i,j,\xi}}(W)$) are independent of $1\le \xi \le \hat g_i$ for each $1\le j \le s_i$, 
we obtain that $(W-L_{n'})E_0\in \gen{\hat g_1, \dots, \hat g_m}$.
On the other hand, 
$L_n=L_{n'}+(M-W)$ by \lemref{l:La} (4).
Therefore,
\[
\deg D_n=-L_nE_0=(W-L_{n'}-M)E_0\in \gen{\hat g_1, \dots, \hat g_m}.
\qedhere
\]
\end{proof}

\begin{cor}
If $g>0$, then $a(R)\in \gen{e_1, \dots, e_m}$ and  $2g-2\in \gen{\hat g_1, \dots, \hat g_m}$.
Note that $a(R)=(m-2)\ell-\sum_{i=1}^me_i$ by \cite[3.1.6]{G-W}.
\end{cor}
\begin{proof}
By \thmref{t:WatD}, $K_{E_0}\sim D_{a(R)}$.  Since $h^0(K_{E_0})=g>0$, the assertion follows from \thmref{t:ehg}.
\end{proof}

\begin{thm}\label{t:g=1MZ}
If  $H^0(D_{\alpha})\ne 0$, then $M_X=Z_X$. 
In particular, if $g\le 1$, then $M_X=Z_X$.
\end{thm}
\begin{proof}
If $H^0(D_{\alpha})\ne 0$, then  $\alpha\in \gen{e_1, \dots, e_m}$ by \thmref{t:ehg}.
Hence  $e_m \le \alpha$, and $M_X=Z_X$ by \thmref{t:BCImain}. 
If $g\le 1$, we have $H^0(D)\ne 0$ for any divisor $D$ on $E_0$ with $\deg D>0$.
\end{proof}

\begin{ex}
We have seen that if $\alpha<e_m$, then $H^0(D_{\alpha})=0$ even though $D_{\alpha}>0$.
We show that the condition $e_m < \alpha$ does not imply $H^0(D_{\alpha})\ne 0$; thus, the converse of \thmref{t:g=1MZ} does not hold.

Suppose that $(a_1,a_2,a_3)=(6,10,45)$.
Then we have 
\[
\{e_{1}, e_2, e_{3}\}=\{15, 9, 2\}, \quad 
\{\hat g_1, \hat g_2, \hat g_3\}=\{5, 3, 2\}, \quad 
\alpha = 3, \quad \deg D_{\alpha}=1,
\]
and  $H^0(D_{\alpha})=0$ by \thmref{t:ehg}.
Note  that the Seifert invariant is $(11, 1, 2(3,1))$.
This is a hyperelliptic type (see \exref{e:mpg}).
Hence $p_g\V=p_g(\Gamma(6,10,45))$.
\end{ex}

\subsection{Non-BCI singularities}
\label{s:NBCI}
In the rest of this section, we assume that $\V$ is an arbitrary weighted homogeneous singularity with resolution graph $\Gamma(a_1, \dots, a_m)$.
We use the same notation as above.
Recall that the Pinkham-Demazure divisor is expressed as $D=Q-\Delta$.

\begin{lem}\label{l:M=Z}
Assume that $\alpha\le e_m$.
Then $M_X= Z_X$ if and only if there exists an effective divisor $F$ on $E_0$ such that $\alpha  D =D_{\alpha}\sim F$ and $\supp F \cap \supp \Delta=\emptyset$.
\end{lem}
\begin{proof}
Let $c=\deg D_{\alpha}$.
Since $\alpha\le e_m$, it follows from \thmref{t:BCImain} and \lemref{l:E_0*} that $Z_X=L_{\alpha}=cE_0^*$ (note that the fundamental cycle is determined by the resolution graph).
On the other hand, $M_X= Z_X$ if and only if there exists a function $h\in H^0(\cO_X(- Z_X))$ such that $\di_X(h)=Z_X+H$, where $H$ is the non-exceptional part. 
In this case, we have $EH=E_0H$ since $H\sim -cE_0^*$.
Thus $(E-E_0)H=0$.
Let $F=H|_{E_0}$. Then $\supp F \cap \supp \Delta=\emptyset$ and 
$D_{\alpha}\sim -L_{\alpha}|_{E_0}\sim F$.

Conversely, suppose that $D_{\alpha}\sim F>0$ and $\supp F \cap \supp \Delta=\emptyset$. 
Since $H^0(D_{\alpha})\ne 0$, there exists $h\in H^0(\cO_X)$ such that $\di_X(h)=cE_0^*+E'+H$ where $E'$ is a cycle supported in $E-E_0$ and $H$ is the non-exceptional part.
By assumption, $(E'+H)|_{E_0}\sim -L_{\alpha}|_{E_0} \sim F$.
In fact, we may assume that $(E'+H)|_{E_0}=F$, since the restriction map $H^0(\cO_X(-L_n)) \to H^0(\cO_{E_0}(D_n))$ is surjective by \thmref{t:TW}.
Then $H|_{ E_0}=F$ by the assumption on the supports, and $E'=0$ since $E'^2=\di_X(h)E'=0$.
\end{proof}

\begin{lem}\label{l:QF} 
For any effective divisor $F\in \Di(E_0)$ such that $\deg F=\deg \alpha D$, there exists a divisor $\t Q\in \Di(E_0)$ such that 
\[
F\sim \alpha \t Q-\alpha \Delta.
\]
Let $\t D= \t Q- \Delta$ and $\t R=R(E_0, \t D)$ (see \sref{ss:star}).
If $R=R(E_0,D)$ is a Gorenstein ring, then $\t R$ is also Gorenstein if and only if $a(\t Q-Q)\sim 0$, where $a=a(R)$.
\end{lem} 
\begin{proof}
Since $\deg(F-\alpha D)= 0$, 
there exists a divisor $Q_F$ with $\deg Q_F=0$ such that $\alpha Q_F\sim F-\alpha D$.  Let $\t Q=Q_F+Q$.
Then 
\[
\alpha \t Q-\alpha \Delta\sim \alpha Q_F+\alpha Q-\alpha \Delta
 \sim  F.
\]
Let $D'$ be the $\Q$-divisor as in \thmref{t:WatD}, and assume that $R$ is Gorenstein. 
Then $K_{E_0}\sim aD-D'$, and $\t R$ is Gorenstein if and only if 
$(aD-D')\sim (a\t D-D')$.
\end{proof}

\begin{thm}\label{t:Z=M}
There exists a weighted homogeneous singularity with resolution graph $\Gamma(a_1, \dots, a_m)$ such that
the maximal ideal cycle coincides with the fundamental cycle on the minimal good resolution.
\end{thm}
\begin{proof}
Let $\V$ be a BCI singularity. If $e_m \le \alpha$, we have $M_X=Z_X$ by \thmref{t:BCImain}.
If $e_m>\alpha$, by \lemref{l:M=Z} and \ref{l:QF}, we can take a  Pinkham-Demazure divisor $\t D$ on $E_0$ so that $\spec R(E_0,\t D)$ satisfies the assertion.
\end{proof}

\section{Examples of singularities in $\cal X(\Gamma(2,3,3,4))$}\label{s:exBCI}

We study some special structures of weighted homogeneous singularities with resolution graph $\Gamma(2,3,3,4)$.
The tuple of integers $(a_1, a_2, a_3, a_4)=(2,3,3,4)$ is characterized by the properties that $a_1+\cdots+a_m$ ($a_i\ge 2$) is minimal such that the corresponding BCI singularity satisfies $E\ne E_0$ and $g=2$.

Let $\Gamma=\Gamma(2,3,3,4)$ and let $\overline{\cal X}(\Gamma)$ denote the set of weighted homogeneous singularities with resolution graph $\Gamma$. 
We shall show that the singularities in $\overline{\cal X}(\Gamma)$ which attain the maximal geometric genus $p_g(\Gamma)$ (see \defref{d:pgG}) are of hyperelliptic type, and obtain the equations for them.
Moreover, we classify the singularities in $\overline{\cal X}(\Gamma)$ with the property that the maximal ideal cycle coincides with the fundamental cycle.

In the following, we use the notation of \sref{s:BCI}.
Notice that the coefficients of the cycles $Z_X$, $L_n$, and $Z_{K_X}$ are determined by $\Gamma$.

First, we give the fundamental invariants of BCI singularities with resolution graph $\Gamma$ (cf. \sref{ss:BCISf}); these data and the following theorem are used in other subsections.

\begin{nota}
Let $\mult\V$ (resp. $\emb\V$) denote the multiplicity (resp. embedding dimension) of the singularity $\V$, namely, that of the local ring $\cO_{V,o}$.\end{nota}

\begin{thm}\label{t:localring}
Let $A:=\cO_{W,p}$ be the local ring of a $d$-dimensional Cohen-Macaulay complex space $W$ at $p\in W$. 
Then we have the following.
\begin{enumerate}
\item {\em (Abhyankar \cite{AbhIneq})} $\emb A \le \mult A + d -1$.
\item {\em (Sally \cite{sally.tangent})} If $A$ is Gorenstein and $\mult A\ge 3$, then  $\emb A \le \mult A + d -2$.
\item {\em (Serre \cite{Serre-projmod})} If $A$ is Gorenstein and $\emb A =d+2$, then $A$ is a complete intersection.
\end{enumerate}
\end{thm}

\subsection{The BCI singularities}
\label{ss:BCI2334}
Assume that $(V,o)$ is a BCI surface singularity with $(a_1, \dots, a_4)=(2,3,3,4)$. Then $V$ can be defined by polynomials
\[
f_1:=x_1^2+x_2^3+p x_3^3, \quad f_2:=x_2^3 +x_3^3+x_4^4 \quad
(p \ne 0, 1).
\]
These are weighted homogeneous of $\deg f_i=\ell=12$ with respect to the weights 
\[
(\deg x_1, \dots, \deg x_4)=(e_{1}, \dots, e_{4})=(6,4,4,3).
\]
We also have $(\alpha_1, \dots, \alpha_4)=(1,1,1,2)$.
By \cite[6.3]{MO},  $\mult(V,o)=a_1a_2=6$.
Let $R=\C[x_1, \dots, x_4]/(f_1, f_2)$. It follows from \cite[3.1.6]{G-W} that
\[
a(R) =12+12-(6+4+4+3)=7.
\]
The Hilbert series of $R$ is 
\begin{equation}\label{eq:BCIH}
H(V,t)=\frac{(1-t^{12})^{2}}{(1-t^{3})(1-t^{4})^2(1-t^{6})}
=1+t^3+2 t^4+2 t^6+2 t^7+3 t^8+\cdots.
\end{equation}
By \proref{p:Hpg} (1),  we have
\[
p_g(V,o)=(2+2 t+2 t^3+t^4+t^7)|_{t=1}=8.
\]

From the result of \sref{ss:BCISf}, we have the resolution graph $\Gamma$ as \figref{fig:2334}.

\begin{figure}[htb]
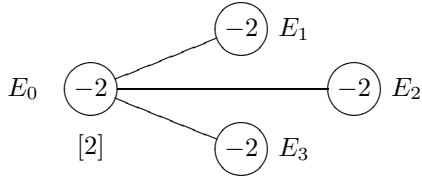

 \begin{center}
$
\xy
(-9,0)*+{E_0}; 
(0,0)*+{-2}*\cir<10pt>{}="E"*++!D(-2.0){[2]}; 
(20,8)*+{-2}*\cir<10pt>{}="A_1"; 
(27,8)*+{E_{1}}; 
(35,0)*+{-2 }*\cir<10pt>{}="B_1"; 
(42,0)*+{E_{2}}; 
(20,-8)*+{-2 }*\cir<10pt>{}="E_1";
(27,-8)*+{E_{3}}; 
\ar @{-} "E" ;"A_1"  
\ar @{-} "E" ;"B_1" 
\ar @{-} "E" ; "E_1" 
\endxy
$
\end{center}
\caption{\label{fig:2334} $\Gamma=\Gamma(2,3,3,4)$}
\end{figure}

Since $\alpha=2<e_4$, we have $Z_X\ne M_X$ by \thmref{t:BCImain}.
In fact, we have that
\[
Z_X=L_2=E+E_0=E_0^*, \quad M_X=Z^{(4)}=L_3=Z_X+E, \quad 
Z_{K_X}=4Z_X.
\]
The {\em fundamental genus} is $p_a(Z_X)=h^1(\cO_{Z_X})=1+Z_X(Z_X+K_X)/2=4$.
The {\em arithmetic genus} of $(V,o)$ is defined by
$p_a(V,o)=\max\defset{p_a(D)}{\text{$D>0$ is a cycle}}$.
It is known that $p_a(Z_X)\le p_a(V,o) \le p_g(V,o)$ (see \cite{wag.ell}). 
By Koyama's inequality (see \cite[Proposition 1.6]{Konno-pf2}), we have $p_a(V,o)=p_a(2Z_X)=5$.  

 The Pinkham-Demazure divisor $D$ and $D_n$ are as follows:
\begin{equation}\label{eq:PD}
D=Q-\sum _{i=1}^3 \frac{1}{2}P_i, \quad 
 D_n=nQ-\sum _{i=1}^3 \Ce{\frac{n}{2}}P_i,
\end{equation}
where $\cO_{E_0}(Q)=\cO_{E_0}(-E_0)$ and $\{P_i\}=E_0\cap E_i$. 
Since $\deg Q=2$, 
we have the following table; these are topological invariant and also used in \sref{ss:maxpg}--\ref{ss:M=Z}.
\begin{center}
$
\begin{array}{c|c|c|c|c|c|c|c}
\hline
n & 1 & 2& 3 & 4 & 5 & 6 & 7 \\
\hline
\deg D_n & -1 & 1 & 0 & 2 & 1 & 3 & 2 \\
\hline
\end{array}
$
\end{center}

The divisor $D$ satisfies the following analytic condition.
\begin{lem}\label{l:Pi}
$Q\sim 2P_i\sim K_{E_0}$ for $i=1,2,3$.
\end{lem}
\begin{proof}
Since $a(R)=7$, by \thmref{t:WatD} and \proref{p:monomials} (2), 
\[
K_{E_0}\sim D_7 \sim D_7-2D_3=Q.
\]
Note that $E_0$ is a hyperelliptic curve with $g=2$.
From \remref{r:ZiHi}, we see that $\{P_1, P_2, P_3\}=\{f_1=f_2=x_4=0\}\subset \P(6,4,4,3)$. Thus, a double cover $E_0\to \P^1$ is given by $(x_1: x_2: x_3: x_4)\mapsto (x_2:x_3)$ and $P_i$ are its ramification points.
Hence $2P_i\sim K_{E_0}$. 
\end{proof}

Later, we shall see the variation of the Pinkham-Demazure divisor $D$ and corresponding singularities with $\Gamma=\Gamma(2,3,3,4)$. 

\subsection{Singularities with $p_g=p_g(\Gamma)$} 
\label{ss:maxpg}

Let $C$ be a nonsingular curve of genus two and $\{P_1, P_2, P_3\}\subset C$ a set of distinct three points.
Let $Q$ be a divisor on $C$ with $\deg Q=2$. 
We define $D$ and $D_n$ ($n\in \Z_{\ge 0}$) as in \eqref{eq:PD}. 
Suppose that $\V\in \overline{\cal X}(\Gamma)$ and the homogeneous coordinate ring $R$ of $(V,o)$ is expressed as 
$R=\bigoplus_{n\ge 0}H^0(D_n)T^n$, where $H^0(D_n)=H^0(C,\cO_C(D_n))$ (see \sref{ss:star}). 
For $n\in \Z_{\ge 0}$, let $R_n=H^0(D_n)T^n$. 
We identify $C$ with the central curve $E_0\subset E$.

\begin{lem}\label{l:gor}
The following are equivalent.
 \begin{enumerate}
\item $(V,o)$ is Gorenstein.
\item $K_{C}$ is linearly equivalent to $D_7$.
\item $h^0(D_7)=2$.
\end{enumerate}
In this case, we have $a(R)=7$.
\end{lem}
\begin{proof}
Since $g=g(C)=2$, for a divisor $F$ of degree $2$ on $C$, $h^0(F)=2$ if and only if $F\sim K_{C}$.
The assertion follows from \thmref{t:WatD}.
\end{proof}

\begin{nota}\label{n:rs}
Let $\cal R(C) \subset C$ be the set of ramification points of the double cover $C\to \P^1$ and $\sigma\: C\to C$ the hyperelliptic involution; we have $\cal R(C)=\defset{P\in C}{\sigma(P)=P}$. 
\end{nota}

From \exref{e:mpg}, we have the following.

\begin{prop}\label{p:maxpg1}
Assume that $P_1\in \cR(C)$, $P_2\in C\setminus \cR(C)$, $P_3=\sigma (P_2)$ and $Q=2P_1$.
Then 
\begin{equation}\label{eq:Dn}
 D_n\sim \begin{cases}
\frac{n}{2}P_1 & (\text{$n$ is even}) \\
\frac{n-3}{2}P_1 & (\text{$n$ is odd})
\end{cases}
\end{equation}
 and $p_g\V=p_g(\Gamma)$.
\end{prop}

We can prove the converse of the above result.

\begin{prop}\label{p:maxpg2}
Assume that $p_g\V=p_g(\Gamma)$.
Then $D$ can be taken as in \proref{p:maxpg1}, namely, by suitable permutation of $P_i$'s, we have $P_1\in \cR(C)$, $P_2\in C\setminus \cR(C)$, $P_3=\sigma (P_2)$, and $Q\sim 2P_1$. 
\end{prop}
\begin{proof}
By \proref{p:Hpg} (2) and Clifford's theorem (cf. \exref{e:mpg}),  we have 
\begin{equation}\label{eq:degDn}
h^0(D_n)=\Fl{\deg D_n /2}+1 \quad \text{if} \quad \deg D_n \le 2.
\end{equation}
Since $\deg D_2=1$ and $h^0(D_2)=1$, there exists a point $P_4\in C$ such that 
\begin{equation}\label{eq:D2}
D_2=2Q-(P_1+P_2+P_3)\sim P_4.
\end{equation}
Since $\deg D_3=0$ and $h^0(D_3)=1$, it follows that
\begin{equation}
\label{eq:D3}
D_3=3Q-2(P_1+P_2+P_3)\sim 0.
\end{equation}
From \eqref{eq:D2} and \eqref{eq:D3}, we have $D_4\sim 2P_4\sim Q$.
Since $h^0(D_4)=2$, we have $P_4\in \cR(C)$.
Therefore, $P+\sigma(P)\sim Q$ for any $P\in C$. 
It follows from \eqref{eq:D2} that
\[
P_1+P_2+P_3 \sim Q+P_4\sim P_1+\sigma(P_1)+P_4.
\]
Hence $P_2+P_3 \sim \sigma (P_1)+P_4$.
If $P_2+P_3 = \sigma (P_1)+P_4$, we are done (e.g., if $P_2=P_4$, then $P_2\in \cR(C)$, $\sigma(P_1)=P_3\not\in \cR(C)$).
If $P_2+P_3 \not= \sigma (P_1)+P_4$, then $h^0(\sigma (P_1)+P_4)=2$, 
and this implies that $P_1=P_4$ and $P_3=\sigma (P_2)$.
\end{proof}

We shall give the fundamental invariants of these singularities.

For an invertible sheaf $\cal L$ on $X$, we say that $P\in X$ is a {\em base point} of $\cal L$ if $\cal L$ is not generated by its global sections at $P$.

\begin{lem}[cf. {\cite[2.7]{wag.ell}, \cite[4.6]{chap}}]
\label{l:multM2}
If $\cO_X(-M_X)$ has no base points, then $\mult\V=-M_X^2$.

\end{lem}

\begin{prop}\label{p:maxpg3}
Assume that $p_g\V=p_g(\Gamma)$. Then we have the following.
\begin{enumerate}
\item $M_X=Z_X+E_1$, where $P_1$ is taken as in \proref{p:maxpg1}.
Furthermore, $\cO_X(-M_X)$ has no base points and $\mult(V,o)=4$.
\item  $p_g\V=10$.
\item $(V,o)$ is a complete intersection singularity defined as 
\[
V=\defset{(x,y,z,w)\in\C^4}{y^2-xz=w^2-h_5(x^2,z)=0},
\]
where $h_5$ is a homogeneous polynomial of degree $5$.
This is a weighted homogeneous singularity of weight type $(2,3,4,10; 6,20)$.
\end{enumerate}
\end{prop}

\begin{proof}
Assume that $D$ is as in \proref{p:maxpg1}.
It follows from \lemref{l:gor} that $\V$ is Gorenstein, because $K_{C}\sim 2P_1\sim D_7$.

(1) Since $h^0(D_2)>0$, there exists a homogeneous function $h\in R_2$ such that $\di_X(h)=Z_X+F+H$, where $F$ is a cycle satisfying $0\le F \le E_1+E_2+E_3$ and $H$ is the non-exceptional part.
Note that any point of $H\cap E$ is in $E_0\setminus \{P_1, P_2, P_3\}$ or $(E_1\cup E_2\cup E_3)\setminus E_0$, because $h$ is homogeneous.
Since 
\[
0\sim \di_X(h)|_{E_0} \sim -D_2+(F+H)|_{E_0}\sim -P_1+(F+H)|_{E_0},
\]
we have $F\cap E_0=\{P_1\}$ and $H\cap E_0=\emptyset$; 
thus $F=E_1$ and $E\cap H\subset E_1\setminus E_0$.
Since  $\cf_{E_1}(L_n)\ge 2$ for all $n\ge 3$, we have that $M_X=Z_X+E_1$ and $\cO_X(-M_X)$ is generated by global sections outside $E_1\cap H$. 
Since $L_4=2E_0^*$ and $D_4\sim 2P_0$ for any $P_0\in \cR(C) \setminus \{P_1\}$, there exists $g\in R$ such that $\di_X(g)=L_4+H'$ where $H'$ intersects $E_0$ only at $P_0$ (cf. the proof of \lemref{l:M=Z}).
Since $\cf_{E_1}(M_X)=\cf_{E_1}(L_4)=2$ and $L_4E_1=0$, 
 $\cO_X(-M_X)$ has no base points. 
Hence $\mult(V,o)=-(M_X)^2=4$ by \lemref{l:multM2}.

(2) Let $(V_0,o)\in \overline{\cal X}(\Gamma)$ be a BCI singularity.
Since $\deg D_n\ge 3$ for $n\ge 8$, $h^0(D_n)$ with $n\ge 8$ is independent of the complex structure of the pair $(C,D)$.
By \eqref{eq:BCIH} and \eqref{eq:degDn}, we  have the Hilbert series $H(V,t)$ of $R=R\V$:
\begin{align}
\begin{split}
\label{eq:maxH}
H(V,t)&=H(V_0,t)+t^2+t^5 
=\frac{\left(1-t^6\right) \left(1-t^{20}\right)}{\left(1-t^2\right) \left(1-t^3\right)
   \left(1-t^4\right) \left(1-t^{10}\right)} \\
&=1+t^2+t^3+2 t^4+t^5+2 t^6+2 t^7+3 t^8+2 t^9+4 t^{10}+\cdots.\end{split}
\end{align}
By \proref{p:Hpg} (2), $p_g\V=p_g(V_0,o)+2=10$.

(3) From \eqref{eq:maxH}, we have the following functions belong to a minimal set of homogeneous generators of $\C$-algebra $R$:
\[
x=f_2T^2\in R_2, \ \ y=f_3T^3\in R_3, \ \ z=f_4T^4\in R_4
\]
such that $\di_{E_0}(f_i)\ge D_i$.
Since $x^3,y^2,xz\in H^0(D_6)T^6$ and $h^0(D_6)=2$, 
we have a relation $r_6(x,y,z)=0$ at degree $6$.
Let $\C[X,Y,Z]$ be the polynomial ring with $(\deg X, \deg Y, \deg Z)=(2,3,4)$.
The difference between the Hilbert series of $R$ and that of the quotient ring $\C[X,Y,Z]/(r_6(X,Y,Z))$ is
\[
H(V,t)-\frac{(1-t^6)}{(1-t^2)(1-t^3)(1-t^4)}
=t^{10}+\cdots.
\]
Hence we have an element $w\in R_{10}$ such that 
$\{x,y,z,w\}$ is a subset of a minimal set of homogeneous generators of $R$.
However, since $\V$ is Gorenstein and $\mult\V=4$, it follows from \thmref{t:localring} that $R$ is a complete intersection generated by just $x,y,z,w$ as $\C$-algebra.
Let $F(t)$ be the Hilbert series of $\C[X,Y,Z,W]/(r_6(X,Y,Z))$, where $\deg W=10$.
Then 
\[
H(V,t)-F(t) =-t^{20}+\cdots.
\]
Hence we have a relation $r_{20}(x,y,z,w)=0$ at degree $20$.
Then the natural $\C$-homomorphism 
\[
S:=\C[X,Y,Z,W]/(r_6(X,Y,Z), r_{20}(X,Y, Z,W)) \to R
\]
induced by $(X,Y,Z,W)\mapsto (x,y,z,w)$ is surjective and the Hilbert series of $S$ coincides with $H(V,t)$.
Hence $S\cong R$.

Next we consider the equations.
Suppose that $\phi\: E_0\to \P^1$ is a double cover such that $\phi(P_1)=\{x_0=0\}$ and $\phi(P_i)=\{x_1=0\}$ ($i=2,3$), where $x_0$ and $x_1$ are the homogeneous coordinates of $\P^1$.
Then $E_0$ can be defined by the equation $x_2^2=x_0h_5(x_0,x_1)$, where $h_5(x_0,x_1)$ is a homogeneous polynomial of degree $5$ such that $h_5(1,0)h_5(0,1)\ne 0$; the branch locus of the covering is $\{x_0h_5(x_0,x_1)=0\}\subset \P^1$.
Now, we can put $x=x_0x_1$, $y=x_0x_1^2$, $z=x_0x_1^3$, $w=x_0^2x_1^5x_2$. Then we have the relations
\[
y^2=x_0^2x_1^4=xz, \ \
w^2=h_5(x_0,x_1)(x_0 x_1^2)^5=h_5(x^2,z).
\qedhere
\]
\end{proof}

\subsection{Singularities with $M_X=Z_X$}
\label{ss:M=Z}
We classify the singularities $(V,o)\in \overline{\cal X}(\Gamma)$ with property that $M_X=Z_X$.
We use the notation of the preceding subsection.

\begin{prop}\label{p:M=Z2334}
We have the following.
\begin{enumerate}
\item $M_X=Z_X$ if and only if there exists a point $P_4\in C\setminus \{P_1, P_2, P_3\}$ such that 
$D_2 \sim P_4$; 
if this is the case,  $D_7\sim 4P_4-Q$. 

\item  Assume that $M_X=Z_X$  and that $x\in R_2$ and $y\in R_m$ belong to a minimal set of homogeneous generators of the $\C$-algebra $R$, where $m$ is the minimum of the degrees of those generators except for $x$.
If $P_4$ is not a base point of $H^0(D_m)$, then $\mult\V=m$.
\end{enumerate}
\end{prop}
\begin{proof}
(1) 
The equivalence follows from \lemref{l:M=Z}. 

(2) We have $\di_X(x)=Z_X+H$, where $H$ is the non-exceptional part. 
Since $H\cap E=\{P_4\}$,  $\cO_X(-Z_X)$ has just a base point $P_4$.
Assume that $u,v$ are the local coordinates at $P_4\in X$ such that $E_0=\{u=0\}$ and $H=\{v=0\}$.
By the assumption, we may also assume that $x=u^2v$ and $y=u^m$.
Note that $m\ge 3$ since $h^0(D_2)=1$. 
Then, at $P_4\in X$,  $\m\cO_X=(u^2v,u^m)\cO_X=(v,u^{m-2})\cO_X(-Z_X)$, where $\m\subset \cO_{V,o}$ is the maximal ideal.
Therefore, the base point of $\cO_X(-Z_X)$ is resolved by the composition $Y\to X$ of $m-2$ blowing-ups at the intersection of the exceptional set and the proper transform of $H$.
Then the maximal ideal cycle $M_Y$ on $Y$ is the exceptional part of $\di_Y(x)$ and by \lemref{l:multM2}, $\mult (V,o)=-M_Y^2=-Z_X^2+(m-2)=m$.
\end{proof}

\begin{rem}\label{r:minmult}
The proof of \proref{p:M=Z2334} shows that $\mult (W,o)\ge -Z_X^2+1=3$ for any normal surface singularity $(W,o)$ with resolution graph $\Gamma$.
\end{rem}

\begin{lem}\label{l:4pts}
Let $P\in C$.
\begin{enumerate}
\item $P\not\in \cR(C)$ if and only if the linear system $|3P|$ is free.
\item There exist distinct three points $A_1, A_2, A_3\in C$ such that $3P\sim \sum_{i=1}^3A_i$. For such points, $P\in \cR(C)$ if and only if $P\in \{A_1, A_2, A_3\}$.
\end{enumerate}
\end{lem}
\begin{proof}
(1) Since $h^0(3P)=2$ by the  Riemann-Roch theorem, $|3P|$ is free if and only if $h^0(2P)=1$.

(2) If the linear system $|3P|$ is free, then the first assertion follows from Bertini's theorem.
If $|3P|$ is not free, then $|2P|=|K_{C}|$ is free and thus we can take distinct three points $A_1:=P, A_2, A_3 \in C$ such that $2P\sim A_2+A_3$.
Suppose that  $3P\sim \sum_{i=1}^3A_i$.
If $P\in \cR(C)$, we have  $P\in \{ A_1, A_2, A_3\}$ since $|3P|$ has a base point $P$.
If $P\in \{A_1, A_2, A_3\}$, then $h^0(2P)=2$. 
\end{proof}

We always assume that  $M_X=Z_X$ in the rest of this section and use the notation above: notice that $h^0(D_2)=1$ and $D_2\sim P_4\in C\setminus\{P_1, P_2, P_3\}$, and that $h^0(D)\ge \deg D-1$ for any divisor $D$ on $C$ by the Riemann-Roch theorem.

Let $H(\Gamma,t)=\sum_{n\ge 0}c_nt^n$ denote the Hilbert series associated with a singularity $(V',o)\in \overline{\cal X}(\Gamma)$ with $p_g(V',o)=p_g(\Gamma)$.  As we have seen in \eqref{eq:maxH},
\[
\sum_{n\ge 0}c_nt^n=1+t^2  +t^3 +2 t^4 +t^5 +2 t^6+2 t^7+\cdots.
\]
We have the following:
\begin{gather*}
h^0(D_n) = c_n 
\text{ for $n=0,1,2,6$ and $n\ge 8$}, \\
h^0(D_3), h^0(D_5) \in \{0,1\}, \quad h^0(D_4), h^0(D_7)\in \{1,2\}.
\end{gather*}
We classify those singularities; they are divided into the following cases:
\begin{enumerate}
\item[I.] $h^0(D_3)=1$.
\item[II.] $h^0(D_3)=0$ and $h^0(D_4)=2$.
\item[III.] $h^0(D_3)=0$ and $h^0(D_4)=1$.
\end{enumerate}
We shall eventually have six cases as seen in Table \ref{tab:M=Z}.

\begin{prop}\label{p:h3=1}
Assume that $M_X=Z_X$. 
If $h^0(D_3)=1$, then $\V$ is not Gorenstein,
$ p_g(V,o)= 8$, $\mult\V=3$, $\emb\V=4$, and
\[
H(V,t)=1+t^2+t^3+t^4+t^5+2 t^6+t^7+\cdots
=\frac{1+t^8+t^{10}}{\left(1-t^2\right) \left(1-t^3\right)}.
\]
Furthermore, the $\C$-algebra $R$ is generated by homogeneous elements of degree $2,3,8,10$.
Note that $\V$ has the minimal multiplicity among the singularities in $\cal X(\Gamma)$ (see \remref{r:minmult}).
\end{prop}
\begin{proof}
We have $h^0(D_5)=1$, since $h^0(D_2)=h^0(D_3)=1$.
Since $D_2\sim P_4$ and $D_3\sim 0$, 
by a similar argument as in the proof of \proref{p:maxpg2} 
we have that 
\[
3Q\sim 2\sum _{i=1}^3 P_i, \quad 
Q\sim 2P_4\sim D_4\sim D_7, \quad 3P_4\sim \sum _{i=1}^3 P_i.
\]
In particular, $h^0(D_4)=h^0(D_7)$.
By \proref{p:M=Z2334} (2), $\mult\V=3$.

Suppose that $h^0(D_4)=2$.
Then $(V,o)$ is Gorenstein by \lemref{l:gor}.
Therefore, $\emb(V,o)\le \mult(V,o)=3$ by \thmref{t:localring}. 
Then $R$ is generated by $x\in R_2$, $y\in R_3$ and $z\in R_4$ as $\C$-algebra $R$ with equation $y^2+xz=0$ (cf. the proof of \proref{p:maxpg3} (3)); however, this implies that $(V,o)$ is rational.
Hence $h^0(D_4)=1$. Then $\V$ is not Gorenstein by \lemref{l:gor}, and therefore $\V$ is not hypersurface.
Thus,  $\emb(V,o)= 4$ by \thmref{t:localring}.
Since $H(\Gamma,t)-H(V,t)=t^4+t^7$, we have $p_g(\Gamma)-p_g\V=2$ by \proref{p:Hpg}.
Since $x,y$ form a regular sequence of $R$, the Hilbert series of $R/(x,y)$ is $H(V,t)(1-t^2)(1-t^3)=1+t^8+t^{10}$.
Then we easily see the degrees of generators.
\end{proof}

\begin{rem}
By \lemref{l:4pts}, we can take distinct points $P_1, \dots, P_4\in C$ such that $3P_4\sim \sum_{i=1}^3P_i$ and $2P_4\not\sim K_{C}$.
Let $Q=2P_4$.
Then we have
\[
D_2 \sim P_4, \quad D_3 \sim 2(3P_4-\sum_{i=1}^3P_i)\sim  0, \quad
h^0(D_4)=h^0(D_7)=h^0(2P_4)=1,
\]
and  $M_X=Z_X$ by \proref{p:M=Z2334}.
Hence we have a singularity $(V,o) \in \overline{\cal X}(\Gamma)$ satisfying all the conditions in \proref{p:h3=1}.
\end{rem}

Next we consider the case $h^0(D_3)=0$.
Since $D_2\sim P_4$,  the following three conditions are equivalent 
(cf. the proof of \proref{p:maxpg2}):
\begin{quotation}
(1) $h^0(D_3)=0$,  \qquad (2)  $3Q\not\sim 2\sum _{i=1}^3 P_i$,
\qquad (3) $Q\not\sim 2P_4$.
\end{quotation}

Let $x\in R_2\setminus \{0\}$.
We will compute the embedding dimension of $\V$ via the curve singularity $(V(x), o)$, where $V(x)=\{x=0\}\subset V$.
Let $H(V(x),t)=\sum_{n\ge 0}d_it^i$ denote the Hilbert series of $R/(x)$.

\begin{lem}\label{l:edC}
The curve $V(x)$ is irreducible and the set $\Gamma_x:=\defset{n\in \Z_{\ge 0}}{d_n\ne 0}$ is a numerical semigroup.
If $\Gamma_x= \gen{m_1, \dots, m_e}$,  then 
\[
\emb\V-1=\emb(V(x),o)\le e. 
\qedhere
\]
\end{lem}
\begin{proof}
Let $H\subset X$ be as in the proof of \proref{p:M=Z2334}.
Then $H$ is irreducible and nonsingular  
 since $EH=1$, and hence the induced map $H\to V(x)$ is the normalization. If $h\in R\setminus (x)$ is a homogeneous element, then the order of $h|_{V(x)}$ at $o\in V(x)$ coincides with the order of vanishing of $h$ along $E_0$, that is, $\deg h$.
Hence $\Gamma_x$ coincides with the so-called {\em semigroup of values} of the curve singularity $(V(x),o)$.
Then the inequality is well-known.
\end{proof}

In the following,  it will be useful to notice that the Frobenius number of $\gen{a,b}$ is $(a-1)(b-1)-1$.

\begin{prop}\label{p:D4=2}
Assume that $M_X=Z_X$. If $h^0(D_3)=0$ and $h^0(D_4)=2$, then $\V$ is not Gorenstein and $\mult(V,o)=4$.
\begin{enumerate}
\item If $h^0(D_5)=1$, then $p_g(V,o)= 8$, $\emb(V,o)=4$, 
\[
H(V,t)=1+t^2+2 t^4+t^5+2 t^6+t^7+\cdots
=\frac{1+t^5+t^{10}+t^{11}}{\left(1-t^2\right)
   \left(1-t^4\right)},
\]
and  $\C$-algebra $R$ is generated by homogeneous elements of degree $2,4,5,11$.

\item If $h^0(D_5)=0$, then $p_g(V,o)= 7$, $\emb(V,o)=5$, 
\[
H(V,t)=1+t^2+2 t^4+2 t^6+t^7+\cdots
=\frac{1+t^7+t^9+t^{10}}{\left(1-t^2\right)
   \left(1-t^4\right)},
\]and  $\C$-algebra $R$ is generated by homogeneous elements of degree $2,4,7,9,10$.
\end{enumerate}
\end{prop}
\begin{proof}
We have that $D_4\sim 2P_4\sim K_{C}$ and  $D_4\not \sim D_7$.
Hence $h^0(D_7)= 1$ and $\V$ is not Gorenstein by \lemref{l:gor}.
Therefore, $\emb\V\ge 4$.
Since $H^0(D_4)$ has no base points, we have $\mult\V=4$ by \proref{p:M=Z2334}, and $\emb(V,o)\le 5$ by \thmref{t:localring}.
Take homogeneous element $y\in R_4$ such that $x$ and $y$ belong to a minimal set of homogeneous generators of $\C$-algebra $R$.
Then $x,y$ form a regular sequence of $R$ and the Hilbert series of $R/(x,y)$ is $H'(t):=H(V,t)(1-t^2)(1-t^4)$.

(1) Assume that $h^0(D_5)=1$. 
We  have $H(V,t)=H(\Gamma,t)-(t^3+t^7)$ and $p_g\V=p_g(\Gamma)-2$ by \proref{p:Hpg} (2).
Since 
\[
H(V(x),t)=H(V,t)(1-t^2)=1+t^4+t^5+t^8\sum_{i\ge 0}t^i,
\]
we have $\Gamma_x= \gen{4,5,11}$.
It follows from \lemref{l:edC} that $\emb\V=4$.
Since $H'(t)=1+t^5+t^{10}+t^{11}$, we obtain the degrees of homogeneous generators of $R$.

(2) Assume that $h^0(D_5)=0$.
Then  $H(V,t)=H(\Gamma,t)-(t^3+t^5+t^7)$,
$H(V(x),t)=1+t^4+t^7\sum_{i\ge 0}t^i$, and $H'(t)=1+t^7+t^9+t^{10}$.
Thus, we obtain the assertion by a similar argument as above.
\end{proof}

\begin{rem}
Let $\cR(C)$ and $\sigma$ be as in \notref{n:rs}.
Suppose that $P_4\in \cR(C)$ and $P_5\in C\setminus\cR(C)$.

(1)
Let $Q=P_4+P_5$. Then $|2Q-P_4|$ is free since $h^0(P_4+2P_5)=2>h^0(2P_5)=h^0(P_4+P_5)$.
Thus, there exist distinct points $P_1, P_2, P_3 \in C\setminus \{P_4\}$ such that $2Q-P_4\sim P_1+P_2+P_3$.
We set $D=Q-\frac{1}{2}\sum_{i=1}^{3}P_i$. Then 
\begin{gather*}
D_2\sim P_4, \quad D_3\sim 2D_2-Q\sim P_4-P_5\not\sim 0, \quad
D_4\sim K_C, \quad \\
 D_5\sim 3D_2-Q\sim 2P_4-P_5\sim (P_5+\sigma(P_5))-P_5=\sigma(P_5).
\end{gather*}
Therefore, we have a singularity satisfying the condition of \proref{p:D4=2} (1).

(2) Let $Q=4P_4-2P_5$. If $|2Q-P_4|$ has a base point $P_0$, then $K_C\sim 2Q-P_4-P_0\sim 7P_4-4P_5-P_0$, and thus $5P_4\sim 4P_5+P_0$.
However, since $|5P_4|$ has a base point $P_4$, we have $4P_4\sim 4P_5$; this is impossible.
Hence $|2Q-P_4|$ is free and there exist distinct points $P_1, P_2, P_3 \in C\setminus \{P_4\}$ such that $2Q-P_4\sim P_1+P_2+P_3$.
Then \begin{gather*}
D_2\sim P_4, \quad D_3\sim 2P_5-2P_4\not\sim 0, \quad
D_4\sim K_C, \quad \\
 D_5\sim  2P_5-P_4, \quad h^0(2P_5-P_4)=0.
\end{gather*}
Hence we have a singularity satisfying the condition of \proref{p:D4=2} (2).
\end{rem}

\begin{prop}\label{p:011}
Assume that $M_X=Z_X$.
If  $h^0(D_3)=0$ and $h^0(D_4)=h^0(D_5)=1$, then  
$\mult(V,o)=\emb(V,o)=5$.
\begin{enumerate}
\item If $h^0(D_7)=2$, then  $(V,o)$ is  Gorenstein,  $p_g(V,o)= 8$, 
\[
H(V,t)=1+t^2+t^4+t^5+2 t^6+2
   t^7+\cdots
=\frac{1+t^6+t^7+t^8+t^{14}}{\left(1-t^2\right)
   \left(1-t^5\right)},
\]
  and  $\C$-algebra $R$ is generated by homogeneous elements of degree 
$2,5,6,7,8$.

\item If $h^0(D_7)=1$, then $(V,o)$ is not Gorenstein, $p_g(V,o)= 7$,
\[
H(V,t)=1+t^2+t^4+t^5+2
   t^6+t^7+\cdots
=\frac{1+t^6+t^8+t^9+t^{12}}{\left(1-t^2\right)
   \left(1-t^5\right)},
\]
  and  $\C$-algebra $R$ is generated by homogeneous elements of degree $2,5,6,8,9$.
\end{enumerate}

\end{prop}
\begin{proof}
The proof is similar to that of \proref{p:D4=2}.
We have $R_4=R_2^2$ and  $D_4\sim 2P_4\not\sim K_{C}$.
Since $D_3\not\sim 0$ and $h^0(D_5)=1$, 
there exists a point $P_5\in C$ such that  $D_5\sim P_5\ne P_4$ (note that $D_2\not \sim D_2+D_3=D_5$). 
Therefore, $\mult(V,o)=5$ by \proref{p:M=Z2334} (2).
Let $y\in R_5\setminus\{0\}$.
Then  the Hilbert series of $R/(x,y)$ is $H'(t):=H(V,t)(1-t^2)(1-t^5)$.
From \lemref{l:gor}, $\V$ is Gorenstein if and only if $h^0(D_7)=2$.

(1)  Assume that $h^0(D_7)=2$.
We have $H(V,t)=H(\Gamma,t)-(t^3+t^4)$ and  
$H'(t)=1+t^6+t^7+t^8+t^{14}$.
Hence $p_g(V,o)=p_g(\Gamma)-2$ by \proref{p:Hpg} and $\emb\V= 5$ by \thmref{t:localring} (2).
 Therefore, $R$ is generated by homogeneous elements of degree $2,5,6,7,8$.

(2) Assume that $h^0(D_7)=1$.
We have $H(V,t)=H(\Gamma,t)-(t^3+t^4+t^7)$, $H'(t)=1+t^6+t^8+t^9+t^{12}$, 
$
H(V,t)(1-t^2)=1+t^5+t^6+t^8\sum_{i\ge 0}t^i,
$
and $\Gamma_x= \gen{5,6,8,9}$.
Hence we obtain the assertion by similar arguments as above.
\end{proof}

The following proposition shows the existence and the property of $D$ corresponding to the singularities in \proref{p:011} (1).

\begin{prop}\label{p:h3=0h7=2}
We have the following.
\begin{enumerate}
\item  There exist points $P_1, \dots, P_4\in C$ and an effective divisor $Q$ of degree two on $C$ which satisfy the condition 
\begin{enumerate}
\item[(C1)]
 $P_1, \dots, P_4$ are distinct, $2Q\sim \sum _{i=1}^4 P_i$, $2P_4\not\sim K_{C}$, $4P_4\sim Q+K_{C}$.
\end{enumerate}
\item Let  $P_1, \dots, P_4$ and $Q$ be as above, and let $D=Q-\frac{1}{2}\sum _{i=1}^3 P_i$.
Then the condition {\rm (C1)} is satisfied if and only if  $M_X=Z_X$ and $h^0(D_3)=0$, $h^0(D_4)=h^0(D_5)=1$, $h^0(D_7)=2$.
\end{enumerate}
\end{prop}
\begin{proof}
(1)
Assume that $\cR(C)$ and $\sigma$ be as in \notref{n:rs}.
Let $P_4\in C$ satisfies $3(P_4-\sigma(P_4))\not\sim 0$.
Then $2P_4\not \sim K_{C}$, because $P_4\not\in \cR(C)$.
Since $\deg(4P_4-K_C)\ge 2$, there exists an effective divisor $Q$ on $C$ such that $4P_4-K_{C}\sim  Q$.
Since $\deg(2Q-P_4)=3$, we have $h^0(2Q-P_4)=2$.
If the linear system $|2Q-P_4|$ is free, then we have distinct three points $P_1, P_2, P_3\in C\setminus \{P_4\}$ such that $2Q\sim \sum _{i=1}^4 P_i$.
If $|2Q-P_4|$ has a base point $G\in C$, then $2Q-P_4-G\sim K_{C}$.
If $G=P_4$, we have $2Q\sim 2P_4+K_{C}$.
Since $4P_4\sim Q+K_{C}$, we have $Q+2P_4\sim 2K_{C} \sim Q+\sigma(Q)$, and hence $2P_4\sim \sigma (Q)$.
However, $4P_4\sim Q+K_{C} \sim \sigma(2P_4)+\sigma(P_4)+P_4$; it contradicts that $3(P_4-\sigma(P_4))\not\sim 0$. 
Therefore, $G\ne P_4$. We can take $P_1\in C$ so that $P_1, P_2:=\sigma(P_1), P_3:=G, P_4$ are distinct.
Then $2Q-P_4\sim K_{C}+P_3\sim P_1+P_2+P_3$. 

(2) 
Assume that (C1) is satisfied.
By \proref{p:M=Z2334} (1), we have $M_X=Z_X$ since $D_2=2Q-\sum _{i=1}^3 P_i\sim P_4$.
We also have
\begin{gather*}
D_3\sim 2P_4-Q \not \sim 0, \ \ 
D_4\sim 2P_4 \not \sim K_{C}, \\
D_5 \sim 3P_4-Q\sim K_{C}-P_4\sim P_4+\sigma(P_4)-P_4=\sigma(P_4), \\ 
D_7\sim 4P_4-Q\sim K_{C}.
\end{gather*}
Thus, we obtain that $(h^0(D_3), h^0(D_4), h^0(D_5), h^0(D_7))=(0,1,1,2)$.

The converse  follows from the arguments above.
\end{proof}

\begin{rem}
We take distinct points $P_4, P_5\in C\setminus \cal R(C)$ such that $P_4+P_5\not\sim K_C$ and $2(2P_4-P_5)\not\sim K_C$, and let $Q=3P_4-P_5$. 
Then $P_4$ is not a basepoint of $|2Q-P_4|$.
As in the proof of \proref{p:h3=0h7=2}, we obtain  distinct points $P_1, P_2, P_3\in C\setminus \{P_4\}$ such that $2Q-P_4\sim P_1+P_2+P_3$.
Then we have
\begin{gather*}
 D_2\sim P_4, \ \ 
h^0(D_3)=h^0(P_5-P_4) =0, \ \ 
h^0(D_4)=h^0(2P_4)=1, \\
h^0(D_5)=h^0(P_5)=1, \ \ 
h^0(D_7)=h^0(P_4+P_5)=1.
\end{gather*}
Hence there exists a singularity satisfying the conditions of  \proref{p:011} (2).
\end{rem}

\begin{prop}\label{p:0101}
Assume that $M_X=Z_X$. If $h^0(D_3)=0$, $h^0(D_4)=1$, $h^0(D_5)=0$. 
Then $\V$ is not Gorenstein, $h^0(D_7)=1$, $p_g(V,o)= 6$, $\mult(V,o)=6$, $\emb(V,o)= 7$,
\[
H(V,t)=1+t^2+t^4+2
   t^6+t^7+\cdots
=\frac{1+t^7+t^8+t^9+t^{10}+t^{11}}{\left(1-t^2\right)
   \left(1-t^6\right)}
\]  and  $\C$-algebra $R$ is generated by homogeneous elements of degree $2,6,7,8,9,10,11$.
\end{prop}
\begin{proof}
Since $D_4\sim 2P_4\not\sim K_C$ and $D_6\sim 3P_4$, $H^0(D_6)$ is free (cf. \lemref{l:4pts}).
Hence we have $\mult\V=6$ by \proref{p:M=Z2334} (2) and $\emb\V\le 7$ by \thmref{t:localring}.
Take a homogeneous element $y\in R_6$ such that $x$ and $y$ belong to a minimal set of homogeneous generators of $\C$-algebra $R$. Then $x,y$ form a regular sequence of $R$ and the Hilbert series of $R/(x,y)$ is $H'(t):=H(V,t)(1-t^2)(1-t^6)$.

If $h^0(D_7)=2$, then $H'(t)=1+2 t^7+t^8+t^{10}+t^{11}-t^{13}+t^{15}$ 
has a negative coefficient; it is a contradiction.
Hence we have $h^0(D_7)=1$.
Then $H(V,t)=H(\Gamma,t)-(t^3+t^4+t^5+t^7)$, 
$H'(t)=1+t^7+t^8+t^9+t^{10}+t^{11}$. 
Hence $p_g\V=p_g(\Gamma)-4$,  $\emb\V=7$  and  $\C$-algebra $R$ is generated by homogeneous elements of degree $2,6,7,8,9,10,11$.
\end{proof}

\begin{rem}
Let $P_4, P_5\in C\setminus \cR(C)$ be distinct points such that $P_4+P_5\not\sim K_C$.
Let $Q=P_4+P_5$.
Then $|2Q-P_4|$ is free because $h^0(P_4+P_5)=h^0(2P_5)=1$.
Hence there exist distinct three points $P_1, P_2, P_3\in C\setminus\{P_4\}$ such that $2Q-P_4 \sim P_1+P_2+P_3$.
Then we have 
\begin{gather*}
 h^0(D_3)=h^0(P_4-P_5)=0, \quad h^0(D_4)=h^0(2P_4)=1,  \\ 
h^0(D_5)=h^0(2P_4-P_5)<h^0(2P_4)=1, \\ 
h^0(D_7)=h^0(3P_4-P_5)<h^0(3P_4)=2.
\end{gather*}
 Therefore, we have a singularity of \proref{p:0101}.
\end{rem}

For reader's convenience, 
we provide a table of the conditions for the Pinkham-Demazure divisors $D=Q-\sum _{i=1}^3 \frac{1}{2}P_i$ which induce the singularities discussed in this subsection;
for each case, $\cR=\cR(C)$, four points $P_1, \dots, P_4\in C$ are distinct, and $P_1+P_2+P_3\sim 2Q-P_4$.

\begin{table}[h]
\renewcommand{\arraystretch}{1.2}
\[
\begin{array}{cccl}
\hline\hline
p_g & \mult & \emb & \text{Pinkham-Demazure divisor} \\
\hline
8 & 3 & 4 &  Q=2P_4, P_4\not\in \cR  \\
 8 & 4 & 4 & Q=P_4+P_5, P_4\in \cR, P_5\not\in \cR \\
 7 & 4 & 5 & Q=4P_4-2P_5, P_4\in \cR, P_5\not\in \cR \\
 8 & 5 & 5 &  Q=4P_4-K_C, P_4\not\in \cR \\
 7 & 5 & 5 & Q=3P_4-P_5, P_4\not\in \cR, P_5\not\in \cR, P_4\ne P_5, \\
&&& P_4+P_5\not\sim K_C, 2(2P_4-P_5)\not\sim K_C \\
 6 & 6 & 7 & Q=P_4+P_5, P_4\not\in \cR,  P_5\not\in \cR, P_4\ne P_5, P_4+P_5\not\sim K_C \\
\hline\hline
\end{array}
\]
\caption{\label{tab:M=ZP-D}
Singularities with $M_X=Z_X$ and Pinkham-Demazure divisors}
\end{table}

\begin{rem}
Taking a general Pinkham-Demazure divisor $D=Q-\sum _{i=1}^3 \frac{1}{2}P_i$, we have a singularity $\V\in  \overline{\cal X}(\Gamma)$ with $H(V,t)=1+ t^4 +2 t^6+ t^7+\cdots$ and that $p_g\V=5$.
Recall that $p_a\V=5$ (see \sref{ss:BCI2334}).
Therefore, we have the equality $p_a(V,o)=\min \defset{p_g(W,o)}{(W,o)\in \cal X(\Gamma)}$, and this is realized by a weighted homogeneous singularity (cf. \thmref{t:TW}).
\end{rem}



\providecommand{\bysame}{\leavevmode\hbox to3em{\hrulefill}\thinspace}
\providecommand{\MR}{\relax\ifhmode\unskip\space\fi MR }
\providecommand{\MRhref}[2]{%
  \href{http://www.ams.org/mathscinet-getitem?mr=#1}{#2}
}
\providecommand{\href}[2]{#2}

\end{document}